\documentclass[12pt]{amsart}
\usepackage{amssymb, amsmath, amsthm}

\newtheorem{theorem}{Theorem}[section]
\newtheorem{cor}[theorem]{Corollary}
\newtheorem{lemma}[theorem]{Lemma}

\theoremstyle{definition}

\newtheorem{rem}[theorem]{Remark}
\newtheorem{conj}[theorem]{Conjecture}

\newtheorem{exam}[theorem]{Example}

\def\beq{\begin{equation}
}
\def\eeq{\end{equation}}

\numberwithin{equation}{section}

\def\beq{\begin{equation}}
\def\eeq{\end{equation}}

\def\cD{ {{\cal D}}}
\def\cH{ {{\cal H}}}

\def\bbR{ {\mathbb R}}

\def\RRx{\RR \langle x \rangle}
\def\RR{ {\mathbb{R}} }

\def\gtupn{(\mathbb{R}^{n\times n}_{sym})^g}

\def\cC{ {\mathcal C} }
\def\cD{ {\mathcal D} }

\def\cH{ {\mathcal H} }

\def\cR{{ \mathcal R }}

\def\beq{\begin{equation}}
\def\eeq{\end{equation}}

\def\smin{ \sigma^{min}  }

\def\tV{\widetilde V}

\def\tpsi{\widetilde{\psi}}

\def\cR{{\mathcal R}}
\def\cZ{{\mathcal Z}}
\def\gZ{{\mathfrak Z}}

\begin{document}

\title[Classification of NC Polynomials]{Classification of all
noncommutative polynomials whose Hessian has negative signature one and
a noncommutative second fundamental form}

\author[Dym]{Harry Dym}
\address{ Department of
Mathematics\\ Weizmann Institute of Science, Rehovot 76100 \\ Israel}

\email{harry.dym@weizmann.ac.il}

\author[Greene]{Jeremy M. Greene ${}^1$}
\address{ Mathematics
Department \\ University of California at San Diego, La Jolla, CA 92093}

\email{j1greene@math.ucsd.edu}


\author[Helton]{J. William Helton${}^1$}
\address{ Mathematics
Department \\ University of California at San Diego, La Jolla, CA 92093}

\email{helton@math.ucsd.edu}

\thanks{{\tiny $1$} Partly
supported by the NSF DMS-0700758 and the Ford Motor Co.}

\author[McCullough]{Scott  A. McCullough${}^2$ }
\address{ Department of Mathematics\\ University of Florida\\
Gainesville, FL  32611-8105}

\email{sam@math.ufl.edu}

\thanks{{\tiny $2$} Partly supported by the NSF grants DMS-0140112 and
DMS-0457504}



\maketitle

\begin{abstract}

Every symmetric polynomial $p=p(x)=p(x_1,\ldots,x_g)$
(with real coefficients) in $g$ noncommuting variables
$x_1,\ldots,x_g$ can be written as a
sum and difference of squares of noncommutative polynomials:
$$
p(x)=\sum\limits^{\sigma_+}_{j=1} f^+_j(x)^T f^+_j(x) -
\sum\limits^{\sigma_-}_{\ell=1}f^-_\ell (x)^T f^-_\ell(x)\,,
\leqno{(SDS)}
$$
where $f^+_j, f^-_\ell$ are noncommutative polynomials. Let
$\sigma^{min}_-(p)$,
the negative signature of $p$, denote the
minimum number of negative squares used in this representation, and let
the \textbf{Hessian} of $p$ be defined by the formula
$$
p^{\prime \prime}(x)[h] := \frac{d^2p(x+th)}{dt^2}|_{t=0}\,.
$$
In this paper we classify
all symmetric noncommutative polynomials $p(x)$ such that
$$
\sigma^{min}_-(p'')\le 1\,.
$$

We also introduce the \textbf{relaxed Hessian} of
a symmetric polynomial $p$ of degree $d$ via the formula
$$
  p^{\prime\prime}_{\lambda,\delta}(x)[h] := p^{\prime \prime}(x)[h] +
   \delta \, \sum m(x)^T h_j^2 m(x) +
  \lambda p^{\prime}(x)[h]^{T} p^{\prime}(x)[h] $$
for $\lambda,\,\delta \in \mathbb{R}$ and
show that if this relaxed Hessian is positive
semidefinite in a suitable and relatively innocuous way,
then $p$ has degree
at most 2. Here the sum is over monomials $m(x)$ in $x$ of degree
at most $d-1$ and $1\le j \le g.$ \\
\indent This analysis is motivated by an attempt to  develop
 properties
of noncommutative real algebraic varieties pertaining to curvature, since,
as will be shown elsewhere,
$$
-\langle p^{\prime\prime}_{\lambda,\delta}(x)[h]v, v\rangle\quad
\textrm{(appropriately restricted)}
$$
plays the role of of a noncommutative second fundamental form.
\end{abstract}




\section{Introduction  and Main Results}
\label{sec:background}
This paper deals with polynomials $p(x)=p(x_1,\ldots,x_g)$ in noncommuting
variables $x_1,\ldots,x_g$. We shall refer to such polynomials as
{\bf nc polynomials}. The first two subsections of this introduction
explain the setting and recall some earlier results that will be needed to
help justify the classification referred to in the abstract. The main results
on nc polynomials are stated in subsection \ref{subsec:MainResults}.
In the noncommutative setting, there is a  natural and
 rigid way  of representing such polynomials
 in terms of a matrix called the middle matrix. The results of subsection
 \ref{subsec:MainResults} on nc polynomials are reformulated
 in subsection \ref{sec:middle} in terms of the
 middle matrix. In fact, the strategy implemented in the
  body of the paper is to establish facts about the middle
  matrix and then deduce corresponding results about polynomials.
  Subsection \ref{sec:motivation} concludes
  the introduction
  with a discussion of  motivation from both
 geometry and engineering, cf.  \cite{HPpreprint}.

\subsection{The setting}
\label{sec:setting}
The setting of this paper coincides with that of \cite{DHMjda}. The principal
definitions are reviewed briefly for the convenience of the reader.

\subsubsection{Polynomials}

Let $x=\{x_1,\dots,x_g\}$  denote noncommuting indeterminates
and let $\RRx$
denote the set of  polynomials $$p(x)=p(x_1,\ldots,x_g)$$
in the indeterminates $x$; i.e., the set of finite linear combinations
\begin{equation}
 \label{eq:genericpoly}
  p=\sum_{|m|\leq d} c_m m \quad\textrm{with}\quad c_m\in \mathbb R
\end{equation}
of monomials (words) $m$ in $x$. The degree of such a polynomial $p$
is defined
as the the maximum of the lengths $\vert m\vert$ of
the monomials $m$ appearing (non-trivially) in the linear combination
(\ref{eq:genericpoly}). Thus,
for example, if $g=3$, then
$$
p_1= x_1 x_2^3 + x_2 + x_3x_1x_2 \quad\textrm{and}\quad
p_2= x_1 x_2^3 + x_2^3 x_1 + x_3x_1x_2 +x_2 x_1 x_3
$$
are polynomials of degree four in $\RRx$.

There is a natural
 {\bf involution} \index{involution} $m^T$ on monomials given by the rule
$$
x_j^T=x_j\quad\textrm{and if}\quad m=x_{i_1}x_{i_2}\cdots x_{i_k},\quad
\textrm{then}\quad
m^T=x_{i_k}\cdots x_{i_2}x_{i_1},
$$
which of course extends to polynomials $p=\sum c_m m$ by linearity:
\begin{equation*}
 p^T =\sum_{|m|\le d} c_m m^T.
\end{equation*}
A polynomial $p\in\RRx$ is said to be {\bf symmetric} \index{symmetric}
if $p=p^T$. The
second polynomial $p_2$ listed above is symmetric, the first is not.
Because of the assumption $x_j^T=x_j$ (which will be in force throughout this
paper) the variables are said to be symmetric too.

\subsubsection{Substituting Matrices for Indeterminates}
\label{sec:openG2}

Let $\gtupn$ denote the set of $g$-tuples
$(X_1,\ldots,X_g)$ of real symmetric $n\times n$ matrices.
We shall be interested in evaluating a
polynomial $p(x)=p(x_1,\ldots,x_g)$
that belongs to $\RRx$ at a
tuple $X=(X_1,\dots,X_g)\in(\mathbb{R}^{n\times n}_{sym})^g$.
In this
case $p(X)$ is also an $n\times n$ matrix and
the involution on $\RRx$
that was introduced earlier is compatible
with matrix transposition, i.e.,
$$
p^T(X)=p(X)^T,
$$
where $p(X)^T$ denotes the transpose of the  matrix  $ p(X)$.
When $X\in \gtupn$ is substituted into $p$
the constant term  $p(0)$ of $p(x)$  becomes $p(0) I_n$.
For example, if $p(x)=3+x^2$, then
$$ p(X)= 3 I_n+X^2.$$

\indent A symmetric nc polynomial $p \in \RRx$ is said to be
{\bf matrix positive}
if $p(X)$ is a positive semidefinite matrix for each tuple
$X=(X_1,\dots,X_g)\in (\mathbb{R}^{n\times n}_{sym})^g$.
Similarly,  a symmetric nc polynomial $p$  is said to be
{\bf matrix convex} if
\begin{equation}
\label{eq:nov13b6}
p(tX+(1-t)Y)\preceq tp(X)+(1-t)p(Y)
\end{equation}
for every pair of tuples $X, Y\in(\mathbb{R}^{n\times n}_{sym})^g$ and
$0\le t\le 1$.

\subsubsection{Derivatives}
We define the directional
derivative of the
monomial $m=x_{j_1}x_{j_2}\cdots x_{j_n}$ as the linear form:
\begin{equation*}
m'[h] = h_{j_1}x_{j_2}\cdots x_{j_n}
       + x_{j_1}h_{j_2}x_{j_3}\cdots x_{j_n} + \ \dots \
       + x_{j_1}\cdots x_{j_{n-1}}h_{j_n}
\end{equation*}
and extend the definition to polynomials $p = \sum c_m m$ by linearity; i.e.,
\begin{equation*}
p^\prime(x)[h]=\sum c_m m'[h].
\end{equation*}
Thus, $p^\prime(x)[h] \in \mathbb R\langle x,h \rangle$ is
the
$$
c\textrm{oefficient of $t$ in the expression}\quad p(x+th)-p(x);
$$
it is an nc
polynomial in $2g$ (symmetric)
variables $(x_1,\dots,x_g,h_1,\dots,h_g).$
Higher order derivatives are computed in the same way: If
$q(x)[h]=h_{j_1}x_{j_2}\cdots x_{j_n}$, then
$$
q^{\prime}(x)[h]=h_{j_1}h_{j_2}x_{j_3}\cdots x_{j_n}
+h_{j_1}x_{j_2}h_{j_3}\cdots x_{j_n}+\cdots+h_{j_1}x_{j_2}\cdots
x_{j_{n-1}}h_{j_n}
$$
and the definition is extended to finite linear combinations of such terms
by linearity.
If $p$ is symmetric, then so is $p^\prime$.
For $g$-tuples of symmetric matrices of a fixed size
$X,H ,$  the evaluation formula
\begin{equation*}
p^\prime(X)[H]=\lim_{t\to 0} \frac{p(X+tH)-p(X)}{t}
\end{equation*}
holds, and  if $q(t)=p(X+tH)$, then
\begin{equation}
\label{eq:nov14a6}
p^\prime(X)[H]=q^\prime(0)\quad\textrm{and}
\quad p^{\prime\prime}(X)[H]=q^{\prime\prime}(0)\,.
\end{equation}
The second formula in (\ref{eq:nov14a6}) is the evaluation of
the {\bf Hessian},
$p^{\prime\prime}(x)[h]$ of a polynomial $p \in \RRx$; it can be
thought of as the formal
second directional derivative of $p$ in the ``direction'' $h$.

If $p^{\prime\prime} \neq 0$, that  is, if $\textrm{degree}\, p \ge 2$,
then the degree of $p^{\prime\prime}(x)[h]$ as a polynomial in the $2g$
variables $(x_1,\ldots,x_g, h_1\ldots,h_g)$ is equal to the degree
of $p(x)$ as a polynomial in $(x_1,\ldots,x_g)$
and is homogeneous of degree two in $h$.
The same conclusion holds for the $k^{th}$ directional derivative
$p^{(k)}(x)[h]$ of $p$  if $k\le d$, the degree of $p$.
The expositions in \cite{HMVjfa} and in \cite{HPpreprint}
give more detail on the derivative and Hessian of
a polynomial in noncommuting variables.

\begin{exam}
A few concrete examples are listed for practice with the definitions, if
the reader is so inclined.
\begin{enumerate}
\item[\rm(1)] If $p(x)=x^4$, then \\$p'(x)[h]= hxxx + xhxx  + xxhx + xxxh$,\\
$p''(x)[h]=
2hhxx   + 2hxhx + 2 hxxh
  + 2xhhx   + 2 xhxh  + 2 xxhh$,
$
\\
p^{(3)}(x)[h]=
 6 (hhhx + hhxh + hxhh +xhhh)$, \\
$
p^{(4)}(x)[h]=
 24 hhhh
$
and $p^{(5)}(x)[h]=0$.
\vspace{2mm}
\item[\rm(2)] If $p(x)= x_2 x_1 x_2$, then $p'(x)[h]=
h_2 x_1 x_2 +x_2 h_1 x_2 + x_2 x_1 h_2. $
\vspace{2mm}
\item[\rm(3)] If $p(x)=x_1^2 x_2$, then $p^{\prime\prime}(x)[h]
= 2(h_1^2 x_2 + h_1 x_1 h_2 + x_1 h_1 h_2)$.
\end{enumerate}
\end{exam}

\subsubsection{The Signature of a Polynomial} \label{subsec:DefSM}
Every symmetric polynomial
$p(x)$ admits a representation of the form
(a sum and difference of squares)
$$p(x) =
\sum\limits^{\sigma_+}_{j=1} f^+_j(x)^T f^+_j(x) -
\sum\limits^{\sigma_-}_{\ell=1}f^-_\ell (x)^T f^-_\ell(x)
\leqno{(SDS)}$$ where $f^+_j, f^-_\ell$ are noncommutative polynomials;
see e.g., Lemmas 4.7 and 4.8 of \cite{DHMjda} for details.
   Such representations are highly non-unique.
However, there is a unique smallest number of positive (resp.
negative squares) $\smin_\pm(f)$ required in an SDS decomposition of
$f$. These numbers will be identified with $\mu_{\pm}(\cZ)$, the number of
positive (negative) eigenvalues of an appropriately chosen symmetric
matrix $\cZ$ in subsection \ref{sec:middle}.

\subsection{Previous results}
\label{sec:prior}
The number (or rather dearth)
of negative squares in an SDS decomposition of
the Hessian places serious
restrictions on the degree of a symmetric nc polynomial.
A theorem of Helton and McCullough [HM04]
states that a symmetric nc polynomial that is matrix convex has degree
at most two.  Since a symmetric nc
polynomial
\begin{equation}
\label{eq:aug27a8}
p\quad \textrm{is matrix convex if and only if} \quad \sigma_-^{min}(p^{\prime\prime})=0,
\end{equation}
(see e.g., \cite{HMe98}) the result in
[HM04] is a  special case of the
following more general result in \cite{DHMjda}.

\begin{theorem}
\label{thm:abs} If $p(x)$ is a symmetric nc polynomial of degree $d$ in
noncommutative symmetric variables $x_1, \ldots, x_g$, then \beq
\label{eq:bigEst} d \le 2 \smin_\pm(p^{\prime\prime}) + 2.
\end{equation}
\end{theorem}

\indent There are refinements which say that if equality holds or is
near to holding in Equation (\ref{eq:bigEst}), then the highest degree
term, $p_d,$  factors or ``partially factors"
in a certain way. Here, and often in what follows, we write
$$
p=\sum_{i=0}^{d}{p_i}\,,
$$
where $p_i$ is either a homogeneous polynomial of degree $i$, or the zero
polynomial if there are no terms of degree $i$ in $p$.

\subsection{Main Results} \label{subsec:MainResults}
In this paper we determine the structure of a
symmetric nc polynomial $p$ when either
\begin{enumerate}
\item[\rm(1)]  $\sigma^{min}_-({p^{\prime\prime}}) \leq 1$; or
\item[\rm(2)]  a variant of the Hessian, described below in
  subsubsection \ref{subsec:relax}, is positive.
\end{enumerate}
The first follows the investigation in \cite{DHMjda} and
the second extends efforts in \cite{DHMind}, \cite{HM04} and
\cite{HMVjfa}. The two are linked together via questions
about the convexity of sublevel sets of symmetric polynomials.

The main conclusions on the structure of $p(x)$ are summarized in
Theorem \ref{thm:mainp}
 and Theorem  \ref{thm:localq} below.

The proofs of these theorems are based on a representation for $p''$
in terms of a matrix $\cZ$ and very detailed analysis of the
structure of $\cZ$. The main structural results on $\cZ$ are discussed in
subsection \ref{sec:middle}.

\subsubsection{The case of at most one negative square}
 \label{subsubsec:onenegsquare}
An nc polynomial $\varphi$ in $g$ variables that is
homogeneous of degree one may be represented by a vector from $\mathbb R^g$;
i.e., associate the vector
$u$ with entries $u_j$ to the polynomial $\varphi =\sum_{j=1}^g u_j x_j$.

Similarly, an nc polynomial $q$ in $g$ variables that is
homogeneous of degree two (i.e., a noncommutative quadratic form)
may be conveniently
represented by a $g\times g$
matrix; i.e., associate the matrix $Q(q)=[q_{ij}]$ to
the polynomial $q=\sum_{i,j} q_{ij}x_ix_j$.

\begin{theorem}
\label{thm:mainp} Suppose $p(x)$ is a symmetric nc polynomial of
degree $d$ in $g$ symmetric variables $x_1, \ldots,
x_g$. Then  $\sigma_-^{min}(p^{\prime\prime}) \leq 1$ \\
if and only if
\begin{equation}
\label{eq:nov18a6X}
p(x)=p_0+p_1(x)+ p_2(x) + \varphi(x)q(x)+q^T(x)\varphi(x)
+\varphi(x) f_0(x) \varphi(x)\,,
\end{equation}
where
\begin{enumerate}
  \item Each nonzero $p_j(x)$ is a homogeneous symmetric polynomial of
degree $j$;
  \item $\varphi(x)=\sum u_j x_j$ is a  homogeneous polynomial
  of degree one
          and $u$ is a unit vector with components $u_1, \ldots, u_g$;
  \item Each of the polynomials $p_2(x)$, $q(x)$ and $f_0(x)$ is either
$0$ or a homogeneous
          polynomial of degree two; and the matrix
$$
 E_2=     \begin{bmatrix} PQ(p_2)P & PQ(q)^T \\ Q(q)P & Q(f_0) \end{bmatrix}
$$
    is positive semidefinite, where $P=I-uu^T$ is the orthogonal projection
    onto the orthogonal complement of $u$.
\end{enumerate}

\medskip

\noindent
If in addition $p$ has degree three, then $f_0=0$ and
$q(x)= f_1(x) \varphi(x)$ for some homogeneous polynomial $f_1(x)$
of degree one.

\medskip

\noindent Finally,
  $\sigma_-^{min}(p^{\prime\prime}) =0$
if and only if $d\le 2$ and $p_2(X)\succeq 0$  for every
$X\in\gtupn$.
\end{theorem}

\begin{proof}
  The proof is in \S \ref{subsec:oct12b8}.
\end{proof}

\subsubsection{The Relaxed Hessian}
\label{subsec:relax}
 In applications to
  convex sublevel sets and the curvature of noncommutative real varieties
 it is useful to introduce a variant of the Hessian
 for  symmetric polynomials $p$ of degree $d$ in $g$ non-commuting
  variables. Let $V_k(x)[h]$ denote the vector of polynomials with entries
$h_jm(x)$, where $m(x)$ runs through the set of $g^k$ monomials of
length $k$,
$j=1,\ldots,g.$ The order of the entries (which is irrelevant for the
moment) is fixed in Section \ref{sec:kronDef} via formula (\ref{eq:vjt}).
Thus, $V_k=V_k(x)[h]$ is a vector of height $g^{k+1}$, and the vectors
\begin{equation}
\label{eq:may14a7}
V(x)[h]=\textup{col}(V_0,\ldots,V_{d-2})\quad \textrm{and}\quad
\tV(x)[h]=\textup{col}(V_0,\ldots,V_{d-1})
\end{equation}
are vectors of height $g\nu$ and $g\widetilde{\nu}$, respectively, where
\begin{equation}
\label{eq:may14b7}
\nu=1+g+\cdots+g^{d-2}\quad \textrm{and}\quad \widetilde{\nu}=
1+g+\cdots+g^{d-1}\,.
\end{equation}
Note that
$$
  \tV(x)[h]^T \tV(x)[h] = \sum_{j=1}^g
  \sum_{|m|\le d-1} \; m(x)^T h_j^2 m(x),
$$
 where $|m|$ denotes the degree (length) of the monomial $m$.

 The {\bf relaxed Hessian} of the symmetric nc polynomial
 $p$ of degree $d$ is the polynomial
\begin{equation}
 \label{eq:relaxedHessian}
   p^{\prime\prime}_{\lambda, \delta}:=
      p^{\prime\prime}(x)[h]+ \delta \, \tV(x)[h]^T \tV(x)[h]
         +\lambda \,  p^{\prime}(x)[h]^T p^\prime(x)[h],
\end{equation}
in which $\lambda,\, \delta\in\RR$.
We say that the {\bf  relaxed Hessian is positive} at
$(X,v)\in \gtupn\times\RR^n$ if
 for each $\delta>0$ there is a $\lambda>0$ so that
$$
   0\le
    \langle p^{\prime\prime}_{\lambda, \delta}(X)[H]v,v\rangle
$$
 for all $H\in\gtupn$. Note that
$$
\lambda_1\ge \lambda\quad\textrm{and}\quad \delta_1\ge \delta \Longrightarrow
\langle p^{\prime\prime}_{\lambda_1, \delta_1}(X)[H]v,v\rangle \ge
\langle p^{\prime\prime}_{\lambda, \delta}(X)[H]v,v\rangle.
$$

The next theorem says that positivity of the relaxed Hessian plus a type
of irreducibility implies $p$ has degree at most two.

\begin{theorem}
 \label{thm:localq}
  Let $p$ be a symmetric nc polynomial of degree $d$ in
  $g$ symmetric variables.
  Suppose
\begin{equation}
\label{eq:may13c7}
n>\frac{1}{2}g\widetilde{\nu}(\widetilde{\nu}-1)\,,
\end{equation}
$X\in\gtupn,$  $v\in\mathbb R^n,$ and
  there is no nonzero polynomial $q$ (not necessarily symmetric)
  of degree less than $d$ such that $q(X)v=0$.
  If the relaxed Hessian is positive at $(X,v)$,
  then  $p$ is convex and has degree at most two.
\end{theorem}

 Theorem \ref{thm:localq} in turn follows from
 the following theorem together with an
 application of the CHSY-Lemma (see Lemma \ref{lem:CHSY}).

\begin{theorem}
 \label{thm:local}
  Let $p$ be a symmetric nc polynomial of degree $d$ in
  $g$ symmetric variables.
  Suppose  $X\in\gtupn$ and  $v\in\mathbb R^n.$
  If the set of vectors $\{\tV(X)[H]v: H\in \gtupn\}$
  has codimension at most $n-1$ in $\mathbb R^{ng\widetilde{\nu}}$
  and if the relaxed Hessian is positive at $(X,v)$,
  then $\sigma_-^{min}(p^{\prime\prime})=0$ and hence
$p$ is convex and has degree at most two.
\end{theorem}

\begin{proof}
 Theorems \ref{thm:localq} and \ref{thm:local} are established in section
\ref{sec:relax}.
\end{proof}

 We end this subsubsection with the following generalization
 of one of the main results of \cite{HM04}.
 The {\bf noncommutative $\epsilon$-neighborhood of $0$}
 is the graded set $\mathcal U=\cup_{n\ge 1}\mathcal U_n$
 with
$$
  \mathcal U_n =\{(X,v): X\in\gtupn, \, \|X\|<\epsilon, \, v\in\mathbb R^n\}
$$
 (where $\|\cdot\|$ is any norm on $\gtupn$).

 Given a (graded) subset $S=\cup_{n\ge 1} S_n$ of
$\cup_{n\ge 1}(\gtupn \times \mathbb R^n)$,
 we say that the {\bf relaxed Hessian is positive on $S$} if
 it is positive at each $(X,v)\in S$.

\begin{cor}
 \label{cor:tomainModHess}
   Let $p$ be a given symmetric nc polynomial of degree $d$ in $g$
   symmetric variables. If there is an $\epsilon>0$
   so that the relaxed Hessian is
   positive on a noncommutative $\epsilon$-neighborhood of $0$,
   then the degree of $p$ is at most two and
   $p$ is convex.
\end{cor}

\begin{proof}
  The corollary is proved in subsection \ref{subsec:oct16a8}.
\end{proof}

\subsubsection{Partial positivity of the Hessian}
  In this section we consider variants
  and corollaries of Theorem \ref{thm:local}
  and its consequence Theorem \ref{thm:localq}
  pertaining to the nonnegativity
  hypotheses on the Hessian of the symmetric polynomial
  $p$  which imply bounds on
  the negative signature $\sigma_-(p^{\prime\prime}).$

  Given a subspace $\mathcal H\subset \gtupn$, we say that
  the {\bf Hessian of $p$ is positive relative to $\mathcal H$}
  at $(X,v)$ if for each $H\in\mathcal H$,
 $$
   0 \le \langle p^{\prime\prime}(X)[H]v,v\rangle.
 $$


\begin{theorem}
 \label{thm:regHess-k}
  Let $p$ be a symmetric nc polynomial of degree $d$ in
  $g$ symmetric variables and let $k$ be a given
  positive integer.
  Suppose  $X\in\gtupn,$ $v\in\mathbb R^n,$ and
  $\mathcal H$ is a subspace of $\gtupn$.
  If the set of vectors $\{V(X)[H]v: H\in \mathcal H \}$
  has codimension at most $kn-1$ in $\mathbb R^{ng\nu}$
  and if the Hessian of $p$ is positive relative
   to $\mathcal H$ at $(X,v)$,
  then $\sigma_-^{min}(p^{\prime\prime})< k$.
  \end{theorem}

 The next theorem follows from Theorem \ref{thm:regHess-k}
 in much the same way that Theorem \ref{thm:localq}
  follows from Theorem \ref{thm:local}.

\begin{theorem}
 \label{thm:regHessq-k}
  Let $p$ be a symmetric  nc polynomial of degree $d$
  in $g$ symmetric variables and fix
  a positive integer $k$. Suppose
\begin{equation}
\label{eq:aug20a7}
n> g\frac{\nu(\nu-1)}{2}
\end{equation}
  $X\in\gtupn,$   $v\in\mathbb R^n,$
  and  there does not exist a nonzero
   polynomial $q$ (not necessarily symmetric)
   of degree less than or equal to
   $d-2$ such that $q(X)v=0.$ If $\mathcal H$ is a subspace
  of $\gtupn$ of codimension at most $kn$ and if
  the Hessian of $p$ is positive relative to $\mathcal H$ at
  $(X,v)$,
  then $\sigma_-^{min}(p^{\prime\prime}) < k+1.$
\end{theorem}

\begin{proof}
The last two theorems are proved  in subsection \ref{sec:pfregHess-k}.
\end{proof}

\begin{rem} \rm
  There are versions  of Theorem \ref{thm:regHess-k} and
  Theorem \ref{thm:regHessq-k} with the relaxed Hessian
  in place of the ordinary Hessian.
\end{rem}

\subsubsection{The modified Hessian}
  The {\bf modified Hessian}
  of a symmetric nc polynomial $p$ is the expression,
$$
 p_\lambda^{\prime\prime}(x)[h]
   =p^{\prime\prime}(x)[h] + \lambda\, p^\prime(x)[h]^T p^\prime(x)[h].
$$
  The {\bf modified Hessian is positive at $(X,v)$} if there is
  a $\lambda >0$ so that
$$
  0 \le \langle p^{\prime\prime}_{\lambda}(X)[H]v,v\rangle\quad\text{for
every}\, H\in\gtupn.
$$

  Note that if the modified Hessian of $p$ is positive at $(X,v)$,
  then so is the relaxed Hessian. Indeed, in that case
  a single  $\lambda$ suffices for all $\delta$.

  The modified Hessian will play an important intermediate
  role between the analysis of the Hessian and relaxed Hessian.

\subsection{The middle matrix}
 \label{sec:middle}
 The proofs and refinements of Theorems \ref{thm:mainp}
 and \ref{thm:localq}
 rest on a careful analysis of a certain representation
 of $p^{\prime\prime}$ which we now introduce.


In \cite{DHMjda} it was shown that if $p(x)$ is a symmetric nc polynomial of
degree $d\ge 2$ in $g$ noncommuting variables, then the Hessian
$p^{\prime\prime}(x)[h]$ admits a representation of the form
\begin{equation}
 \label{eq:defs-middle}
   \begin{split}
   p^{\prime\prime}(x)[h]=& V(x)[h]^TZ(x)V(x)[h]  \\
   =& [V^T_0, V^T_1, \ldots, V^T_\ell]
  \left[\begin{array}{ccccc}
  Z_{00}&Z_{01}& \cdots&Z_{0, \ell-1}& Z_{0\ell}\\
  Z_{10}&Z_{11}& \cdots&Z_{1, \ell-1}&0\\
  \vdots&\vdots&&\vdots&\vdots\\
  Z_{\ell 0}&0& \cdots&0&0\end{array}
  \right]
  \left[\begin{array}{c}
  V_0\\V_1\\ \vdots\\ V_{\ell}\end{array}\right],
 \end{split}
\end{equation}
in which $\ell=d-2$, $V(x)[h]$ is the border vector with vector components
$V_j(x)[h]$ of height $g^{j+1}$,
 and $Z(x)=[Z_{ij}(x)]$, $i, j=0,\ldots,d-2$,
the {\bf middle matrix},
is a symmetric matrix  polynomial with matrix polynomial entries $Z_{ij}(x)$
of size
$g^{i+1}\times g^{j+1}$ and degree no more than $(d-2)-(i+j)$
for $i+j\le d-2$ with
$Z_{ij}(x)=0$ for $i+j> d-2$.
Since $p$ is symmetric $Z_{ij}=Z_{ji}^T$ and, since $p$ has
degree $d$,   $Z_{ij}$ is constant when $i+j=d-2$.

Let $\mathcal Z=Z(0)$ for
$0\in\mathbb R^g.$ We call $\mathcal Z$ the {\bf scalar middle matrix}.
The main conclusions  from \cite{DHMjda} that are
relevant to this paper are:
\begin{enumerate}
\item[\rm(1)] $Z(x)$ is polynomially congruent to the
   scalar middle  matrix ${\cZ}=Z(0)$, i.e., there exists a matrix polynomial
$B(x)$ with an inverse $B(x)^{-1}$ that is again a matrix polynomial such that
\begin{equation}
\label{eq:may13a7}
Z(x)=B(x)^TZ(0)B(x)=B(x)^T\cZ B(x)\,.
\end{equation}
\vspace{2mm}
\item[\rm(2)] $\mu_{\pm}({\cZ})=\smin_\pm(p^{\prime\prime}(x)[h])$.
\vspace{2mm}
\item[\rm(3)] If  $X\in\gtupn$, then
\begin{equation}
\label{eq:apr30a7}
\mu_\pm (Z(X)) =n\mu_\pm(\cZ)\,.
\end{equation}
\item[\rm(4)] The degree $d$ of $p(x)$ is subject to the bound
\begin{equation}
\label{eq:nov10a6}
d\le 2\mu_{\pm}({\cZ})+2\,.
\end{equation}
\item[\rm(5)] If $p(x)=p_0(x)+\cdots+p_d(x)$ is expressed as a sum of
homogeneous polynomials $p_j(x)$ of degree $j$ for $j=0,\ldots,d$, then
the homogeneous
components $p_k$ of $p$ with $k\ge 2$ can be recovered from
$\cZ$ by the formula
\begin{equation}
\label{eq:nov10b6}
p_{i+j+2}(x)=\frac{1}{2}[x_1,\ldots,x_g]_{i+1}{\cZ}_{ij}
([x_1,\ldots,x_g]_{j+1})^T
\end{equation}
for $i+j\le d-2,$
\end{enumerate}
where $[x_1,\dots,x_g]_j$ denotes the $j$-fold Kronecker product that is
defined in  (\ref{eq:may6a7}); it is a row vector with $g^j$ entries
consisting of
all the monomials (in $x$) of degree $j$.

If $\mu_-({\cZ})\le 1$, then the bound (\ref{eq:nov10a6})
implies that $d\le 4$
and consequently that ${\cZ}$
will be a matrix of the form
\begin{equation}
\label{eq:nov10c6}
{\cZ}=\begin{bmatrix}{\cZ}_{00}&{\cZ}_{01}&{\cZ}_{02}\\
{\cZ}_{10}&{\cZ}_{11}&0\\{\cZ}_{20}&0&0\end{bmatrix}\,.
\end{equation}

The next theorem, which describes the structure of $\cZ$ when
$\mu_-(\cZ)\le 1$, requires some additional notation.
Viewing a $g\times g^2$ matrix such
as $Z_{01}$, as a $g\times g$ block matrix $C=[c_{ij}]$ with entries
$c_{ij}\in\mathbb{R}^{1\times g}$,
the {\bf structured transpose} $C^{sT}$ of
$C$ is the $g\times g$ block matrix with $ij$
entry equal to $c_{ji}$.  Thus, for example, if $g=2$ and
$$
C=\begin{bmatrix}c_{11}&c_{12}\\c_{21}&c_{22}\end{bmatrix},\quad\text{then}
\quad C^{sT}=\begin{bmatrix}c_{11}&c_{21}\\c_{12}&c_{22}\end{bmatrix}\quad
\textrm{and}\quad C^T=
\begin{bmatrix}c_{11}^T&c_{21}^T\\c_{12}^T&c_{22}^T\end{bmatrix}.
$$
A number of identities involving this definition are presented in Lemmas
\ref{lem:jun3a7}--\ref{lem:aug14c8}.

Given a $g\times n$ matrix $A,$
expressed in terms of its columns as
$$
A=\begin{bmatrix} a_1 & a_2 & \dots & a_n \end{bmatrix}\,, \quad\textrm{let}
\quad \mbox{vec}(A)=\begin{bmatrix} a_1 \\ a_2 \\ \vdots \\ a_n \end{bmatrix}
\,,
$$
and, conversely, given a vector $w\in \mathbb R^{gn}$ with entries
$a_1,\ldots,a_n\in\mathbb R^g$,
$$
 w = \begin{bmatrix} a_1 \\ a_2 \\ \vdots \\ a_n \end{bmatrix}\,,\quad
\textrm{let}\quad \mbox{mat}_g(w)=
\begin{bmatrix} a_1 & a_2 & \dots & a_n \end{bmatrix}\,.
$$
Thus, if $A$ and $w$ are as above, then
$$
\textrm{mat}_g\,\textrm{vec}\,A=A\quad\textrm{and}\quad
\textrm{vec}\,\textrm{mat}_g\,w=w\,.
$$

\begin{theorem}
\label{thm:mainZ}
Let $\cZ$ denote the scalar middle matrix of the Hessian of a symmetric
nc polynomial $p(x)$.
\begin{enumerate}
\item[\rm(I)] If $\textup{degree}\, p\ge 3$, then $\mu_-(\cZ)= 1$ if and
only if
\begin{equation*}
   \cZ=
\begin{bmatrix}
{\cZ}_{00} & {\cZ}_{01} & {\cZ}_{02} \\
{\cZ}_{10} & {\cZ}_{11} & 0 \\
{\cZ}_{20} &      0     & 0
\end{bmatrix}
\end{equation*}
where
\begin{itemize}
 \item[\rm(1)] $\cZ_{01}= u (u^T\otimes y^T)+ uv^T +(uv^T)^{sT}=\cZ_{01}^{sT}$
and $\cZ_{10}=\cZ_{01}^T$;
 \item[\rm(2)] $\cZ_{11}= (u \otimes I_g) A (u \otimes I_g)^T$; and
 \item[\rm(3)] $\cZ_{02} = u(u^T \otimes w_A^T)$; and $\cZ_{20}=\cZ_{02}^T$,
\end{itemize}
in which
\begin{itemize}
 \item[(i)]  $u,y\in \bbR^g$ and $u$ is a unit vector;
 \item[(ii)] $v\in \mathbb R^{g^2}$
   and $(\textup{mat}_g\,v)u=0$.
\item[(iii)] $A=A^T\in\RR^{g\times g}$ and $w_A=\textup{vec}(A)\in\RR^{g^2}$.
\item[(iv)] The matrix
$$
      E_1=\begin{bmatrix} P \cZ_{00} P & (\textup{mat}_g v)^T \\
      \textup{mat}_g v & A \end{bmatrix}\quad(\text{with}\quad P=I-uu^T\quad
\text{and}\quad A=U^T\cZ_{11}U)
   $$
    is positive semidefinite.
\end{itemize}
Moreover, if $\textup{degree}\, p=3$, then
in addition to the  conditions stated above, $A=0$, $w_A=0$, and
 $v=0$.
\item[\rm(II)] If $\textup{degree}\, p=2$, then $\mu_-(\cZ)= 1$ if and
only if $\cZ_{ij}=0$ for $i+j\ge 1$ and there exists a vector
$u\in\RR^g$ with $u^Tu=1$ such that $\cZ_{00}u=\lambda u$ with $\lambda <0$
and $P\cZ_{00}P\succeq 0$ for $P=I_g-uu^T$.
\item[\rm(III)] $\mu_-(\cZ)=0$ if and only if $\cZ_{ij}=0$ for $i+j\ge 1$
and $\cZ_{00}\succeq 0$.
\end{enumerate}
\end{theorem}

\begin{proof}
See subsection \ref{subsec:oct12a8}.
\end{proof}

Theorem \ref{thm:mainp} will follow from Theorem \ref{thm:mainZ} and
the results in section \ref{sec:Zneg1}.

\bigskip

The next theorem gives the explicit correspondence between
Theorems \ref{thm:mainp} and \ref{thm:mainZ}.

\begin{theorem}
\label{thm:mainZsp}
The entries of $\cZ$ in Theorem \ref{thm:mainZ}
correspond to terms of   $p$ in Theorem \ref{thm:mainp}
as follows
\begin{enumerate}
  \item
  $\varphi(x)= [x_1 \ \cdots \ x_g] u $
\item
  $q(x)= \frac{1}{4}(u\otimes y + 2v)^T ([x_1 \ \cdots \ x_g]_2)^T $
  \item
  $f_0(x)=  \frac{1}{2}x^T A x $
\end{enumerate}
\end{theorem}

\begin{proof}
See subsection \ref{subsec:oct12c8}.
\end{proof}

\begin{rem}
If $p_4(x)=\frac{1}{2}\varphi(x)[x_1\ \cdots \ x_g]A[x_1\ \cdots \ x_g]^T
\varphi(x)$, then, since
$$
p_4(x)=\frac{1}{2}[x_1 \ \cdots \ x_g]_2\cZ_{11}([x_1 \ \cdots \ x_g]_2)^T=
\frac{1}{2}[x_1 \ \cdots \ x_g]\cZ_{02}([x_1 \ \cdots \ x_g]_3)^T,
$$
the formulas in Theorem \ref{thm:mainZ} imply that
$$
\cZ_{11}=UAU^T\quad\text{and}\quad \cZ_{02}=u(u^T\otimes w_A^T)\ \textup{with}
\ w_a=\textup{vec} A.
$$
Thus,
$\textup{rank}\cZ_{11}$ can vary between $1$ and $g$, when
$\textup{rank}\cZ_{02}=1$.
\end{rem}

It is not difficult to see that if  $q(x)[h]$ is a
symmetric nc polynomial in the variables $x$ and $h$ and if $q$ is homogeneous
of degree two in $h$, then it has a representation like that
in equation (\ref{eq:defs-middle}). In particular, the modified
Hessian can be expressed as
\begin{equation}
 \label{eq:modhess-middle}
  p^{\prime\prime}_{\lambda}(x)[h]
    = \widetilde{V}(x)[h]^T Z_{\lambda}(x) \widetilde{V}(x)[h].
\end{equation}
  Note here that the vector $\widetilde{V}(x)[h]$ is now of degree $d-1$ in
  $x$ and homogeneous of degree one in $h$.

The proof of Theorem \ref{thm:localq} relies on the following
structure theorem for $\cZ_\lambda$, the {\it scalar middle  matrix of
the modified
Hessian of $p$.}

\begin{theorem}
\label{thm:absModHess}
If $p(x)$ is a symmetric polynomial of degree $d$ in noncommutative symmetric
variables $x_1, \ldots, x_g$, then
the middle matrix $Z_\lambda(x)$ for the  modified Hessian for
$p(x)$ is polynomially congruent to the constant matrix
\begin{equation*}
{\cZ}_\lambda=\begin{bmatrix} {\cZ}&0\\0&
\lambda\psi_{d-1}(0)\psi_{d-1}(0)^T
\end{bmatrix}
\end{equation*}
in which $\psi_{d-1}(0)$ is a nonzero vector of size $g^d$. Moreover, the
congruence is independent of $\lambda$
and consequently, if $\lambda>0$, then
\begin{enumerate}
\item[\rm(1)] $\mu_+({\cZ}_{\lambda})=\mu_+({\cZ})+1$ and
$\mu_-({\cZ}_{\lambda})=\mu_-({\cZ})$;
\vspace{2mm}
\item[\rm(2)] ${\cZ}_{\lambda}\succeq 0\Longleftrightarrow {\cZ}\succeq 0$;
\vspace{2mm}
\item[\rm(3)] ${\cZ}_{\lambda}\succeq 0\Longrightarrow d\le 2$; and
\vspace{2mm}
\item[\rm(4)] $d \le 2\smin_-(p_\lambda^{\prime\prime}(x)[h]) + 2. $
\end{enumerate}
\end{theorem}

\begin{proof}
This theorem is established in Section \ref{sec:iMHintro}.
\end{proof}

\subsection{Motivation} \label{sec:motivation}

The theorems in this paper on the relaxed Hessian
 bear on
a   conjecture concerning the convexity of the positivity set
$$
\cD_p:= \cup_{n \geq 0} \cD_p^n\,,
$$
of a symmetric nc polynomial $p$, where
 $\cD_p^n$ is the connected component of

$$
\mathcal P_p^n:= \{ X= \{X_1, \ldots, X_g \} :
X_j \in \RR^{\,n \times n}_{sym}\ \textup{and} \ p(X)\succ 0\}
$$
containing $0$.
An example of a ``convex'' positivity domain $\cD_p$,
is a ball, namely a set of the form $\cD_q$
where $q$ has the form
$$
q(x)= c - \sum_j^w (a_j + L_j(x) )^T (a_j + L_j(x) )
$$
where $a_j, c$ are real numbers and $L_j(x)$
are linear in $x_1, \dots, x_g$, that is,
$L(x)= b_1 x_1 + \cdots + b_g x_g$.

A (loosely stated)conjecture of Helton and McCullough is
\begin{conj}
\label{conj:convS}
(Helton and McCullough)  If $p$
is ``irreducible'' and $\cD_p$ is convex,
then $\cD_p$ is a ball.
\end{conj}

\subsubsection{Noncommutative real algebraic geometry}
One might think of Conjecture \ref{conj:convS}
and its supporting results  as the beginnings of a
noncommutative real algebraic geometry
in that this is like supposing
``if the boundary  $\partial \cD_p$ of $\cD_p$ has positive curvature''
and concluding
that ``it has constant curvature.''

Since
$$
\gtupn=\RR^g\quad\textrm{when}\ n=1,
$$
it is reasonable to turn to  functions $f$ on $\RR^g$ (now considered
with commuting
variables) for guidance.
In particular smooth quasiconvex functions $f$ on
$\RR^g$ are functions  whose sublevel sets
$$\cC_c :=  \{ X: \ f(X) \leq c \}$$
are convex. It can be shown that a set $\cC_c$ with boundary
$\partial\cC_c$ and  gradient $(\nabla f)(X) \neq 0$ for $X
\in \partial\cC_c$ is convex if and only if {\it the Hessian of $f$
restricted to the tangent plane of $f$ at $X$ is positive semidefinite.}
This is the same as the {\it second fundamental form} of
differential geometry
being nonnegative.
In the terminology of this paper for $n=1$ this condition for $X\in\RR^g$ is
that $p''(X)[H] \succeq 0$ for all $H\in\RR^g$ satisfying  $p'(X)[H] =0$.
In the scalar case ($n=1$) $v$ is effectively 1, so we do not write it.
An easy Lagrange multiplier argument shows
that this is equivalent to {\it there is a $\lambda >0$ making the
modified Hessian, $p^{\prime\prime}_\lambda(X)[H]$,
nonnegative for all $H$.}
Thus a smooth function $f$  with nowhere vanishing gradient (except
at its global minimum) is
quasiconvex on a bounded domain $B\subset \mathbb R^g$
if and only if for $\lambda$
large enough $p''_\lambda(X)(X)[H]$ is nonnegative
for all $X \in B\subset \RR^g$ and all $H \in  \RR^g$.

The papers \cite{DHMind} and subsequent work
represent progress on Conjecture \ref{conj:convS} for the
{\bf noncommutative case}.
We find that an appropriately restricted  quadratic form based on
the relaxed Hessian plays the role of the second fundamental form very
effectively and seems to be the key to
understanding ``noncommutative curvature.''
Its power emanates to a large extent from the classification
theorems established in this paper.

In a different direction we mention that
many engineering problems connected with linear systems lead to
constraints in which a
polynomial or a matrix
of polynomials with matrix variables is
required to take a positive semidefinite value.
Many of these problems are also
``dimension free" in the sense that the
form of the polynomial remains the same for matrices of all sizes.
This is one very good reason to study noncommutative polynomials.
If these polynomials are ``convex" or
``quasiconvex" in the unknown matrix variables, then
numerical calculations are much more reliable.
 This motivated Conjecture \ref{conj:convS}  and our results
 suggest that  matrix convexity
  in conjunction with a type of ``system scalability''
 produces incredibly heavy constraints.

\subsubsection{Noncommutative analysis} This article could also come
under the general heading of ``free analysis",
since the setting is a noncommutative algebra whose generators
are ``free" of relations. This is a
burdgeoning area, of which free probability,
invented by Voiculescu \cite{Vo05} and \cite{Vo06} is currently the
largest component. The interested reader is referred to the web
site \cite{AIMfree06}
of the American Institute of Mathematics, in particular it gives the
findings of the AIM workshop in 2006 on free analysis.
Excellent work in other directions with free algebras has been carried out
by Ball et al,
\cite{Ball06}, Kaluzny-Verbotzky-Vinnikov, e.g.  \cite{VVV06},
Popescu, e.g. \cite{Po06}, and by Voiculescu \cite{Vo05} and \cite{Vo06}.
A fairly expository article describing noncommutative convexity,
noncommutative semialgebraic geometry and relations to engineering is
\cite{HPpreprint}.
\bigskip

\noindent
{\bf Acknowledgement} The authors extend their thanks to the referee for
his/her careful reading of the paper, and for catching a number of misprints
and
making useful suggestions for improving
the clarity of the presentation.

{\section{Kronecker Product Notation}\footnote{This section is taken
from \cite{DHMjda} for the convenience of the reader.}}
 \label{sec:kronDef}
Recall the {\bf Kronecker product}
\begin{equation}
 \label{eq:kronecker}
A\otimes B = \left[\begin{array}{ccc}
a_{11} B&\cdots&a_{1q}B\\ \vdots&&\vdots\\
a_{p1}B&\cdots&a_{pq}B\end{array}\right]
\end{equation}
of a pair of matrices $A\in \mathbb{R}^{p\times q}$
and $B\in \mathbb{R}^{s\times t}.$

A number of formulas that occur in this paper are
conveniently  expressed  in
terms of $j$-fold iterated
Kronecker products of row or column vectors of noncommuting variables.
Accordingly we introduce the notations,
\begin{equation}
\label{eq:may6a7}
[x_1 \ \cdots \ x_g]_j=[x_1 \ \cdots \ x_g]\otimes \cdots \otimes
[x_1 \ \cdots \ x_g]\quad\textrm{($j$ times)}
\end{equation}
and
\begin{equation}
\label{eq:may6b7}
\begin{bmatrix}x_1 \\ \vdots \\ x_g\end{bmatrix}_j=\begin{bmatrix}x_1 \\
\vdots \\ x_g\end{bmatrix}\otimes \cdots \otimes
\begin{bmatrix}x_1 \\ \vdots \\
x_g\end{bmatrix}\quad\textrm{($j$ times)}\,.
\end{equation}

\subsection{Border vectors}
In this section we shall expand upon the representation of the Hessian
$p^{\prime\prime}(x)[h]$ in terms of a middle matrix $Z(x)$ and the
components
$V_j=V_j(x)[h]$ of the border vector $V(x)[h]$
defined by the formulas
\begin{equation}
\label{eq:vjt}
V_j^T(x)[h]=\left\{\begin{array}{ll} &[h_1 \ \cdots \ h_g]\quad\textrm{if}
\quad j=0\,,\\
&[x_1 \ \cdots \ x_g]_j\otimes [h_1 \ \cdots \ h_g]\quad
\textrm{if}\quad j=1,2\ldots\,,\end{array}\right.
\end{equation}

Thus, for example, the
components
$[V^T_0 \ V^T_1\  \cdots\  V^T_{d-2}]$ of the transpose $V^T$
of the border vector $V$ for polynomials of two noncommuting
variables may be written
\begin{eqnarray*}
V^T_0(x)[h]&=& [h_1\quad h_2],\\
V^T_1(x)[h]&=& [x_1\quad x_2 ]\otimes [h_1\quad h_2],\\
V^T_2(x)[h]&=&[x_1\quad x_2 ]\otimes [x_1\quad x_2 ]
\otimes  [h_1\quad h_2],\\
&& \vdots\\
\end{eqnarray*}
In this ordering,
$$
 V^T_1(x)[h] = [x_1h_1\quad x_1 h_2\quad x_2 h_1\quad x_2 h_2]
$$
 and
$$
V^T_2(x)[h] =[x_1x_1h_1\ x_1 x_1h_2\ x_1x_2 h_1\ x_1x_2 h_2\ x_2 x_1h_1\
x_2 x_1h_2\ x_2
x_2 h_1\ x_2 x_2 h_2].
$$

However, since we are dealing with noncommuting variables,
{\it the rules for extracting $V_j$ from $V^T_j$
are a little more complicated than might
be expected}. Thus, for example, the given formula for $V_1^T$ implies that
$$
V_1 = \left[\begin{array}{c} h_1x_1 \\ h_2x_1 \\ h_1x_2 \\ h_2x_2
\end{array}\right] \ne
\left[\begin{array}{c}h_1\\h_2\end{array}\right]\otimes
\left[\begin{array}{c}x_1\\x_2\end{array}\right]\,,
$$
as one might like. Nevertheless, the situation is not so bad,
because the
left hand side is just a permutation of the right hand side.
This propagates:

\begin{equation}
\label{eq:vja}
\left[\begin{array}{c}h_1\\\vdots\\h_g\end{array}\right]
=V_0\quad\textrm{and}
\quad
\left[\begin{array}{c}h_1\\\vdots\\h_g\end{array}\right]\otimes
\left[\begin{array}{c}x_1\\\vdots\\x_g\end{array}\right]_j=\Pi_jV_j\quad
\textrm{for}\quad j=1,2,\ldots
\end{equation}
for suitably defined permutation matrices $\Pi_j$, $j=1,2,\ldots\,$.
The permutation matrix $\Pi_j$ in formula (\ref{eq:vja}) can also be
characterized by the condition
\begin{equation}
\label{eq:nov6d6}
\left[\begin{array}{c}x_1\\\vdots\\x_g\end{array}\right]_{j+1}=\Pi_j
(\left[\begin{array}{ccc}x_1&\cdots&x_g\end{array}\right]_{j+1})^T
\quad
\textrm{for}\quad j=1,2,\ldots
\end{equation}
 The second set of formulas in (\ref{eq:vja})
can also be written in the form
\begin{equation}
\label{eq:vj}
\Pi_jV_j(x)[h]=\left(\left[\begin{array}{c}h_1\\\vdots\\h_g\end{array}\right]
\otimes I_{g^j}\right)\left[\begin{array}{c}x_1\\\vdots\\x_g\end{array}
\right]_j\,,\quad j=1,2,\dots\,.
\end{equation}
The general formula
\begin{equation}
\label{eq:oct1a}
\left([u_1,\ldots,u_k]\otimes [v_1,\ldots,v_{\ell}]\right)a=
\left([u_1,\ldots,u_k]\right)A\left([v_1,\ldots,v_{\ell}]^T\right)
\end{equation}
with
\begin{equation}
\label{eq:oct1b}
a^T=[a_1, a_2,\ldots,a_{k\ell}]\quad\textrm{and}\quad A=\begin{bmatrix}
a_1&\cdots &a_{\ell}\\a_{\ell+1}&\cdots &a_{2\ell}\\\vdots & &\vdots\\
a_{(k-1)\ell+1}&\cdots &a_{k\ell}\end{bmatrix}
\end{equation}
is useful, as is the identity
\begin{equation}
\label{eq:nov6a}
\begin{bmatrix}u_1\\ \vdots\\ u_k\end{bmatrix}\otimes
\begin{bmatrix}v_1\\ \vdots\\ v_{\ell}\end{bmatrix}=
\left(\begin{bmatrix}u_1\\ \vdots\\ u_k\end{bmatrix}\otimes I_{\ell}\right)
\begin{bmatrix}v_1\\ \vdots\\ v_{\ell}\end{bmatrix}\,.
\end{equation}
\subsection{Tracking derivatives}
The Kronecker notation is useful for keeping track of the positions
of the variables $\{h_1,\ldots,h_g\}$ when computing the
derivatives $p^{\prime}(x)[h]$ and
$p^{\prime\prime}(x)[h]$ of the polynomial $p(x)$ in the direction $h$.
This is a consequence of two facts:
\begin{enumerate}
\item[\rm(1)]  Every polynomial is a sum of
homogeneous polynomials.
\item[\rm(2)]  Every homogeneous polynomial
$p=p(x)=p(x_1, \ldots , x_g)$ of degree
$d$ in $g$ noncommuting variables can be expressed 
as a $d$-fold Kronecker
product times a vector $a$  with  $g^d$ real entries:
$$
\textrm{$p$ homogeneous of degree $d$} \Longrightarrow
p(x)=([x_1\ \cdots \ x_g]_d)a\,.
$$
\end{enumerate}
Thus, for example,
if
$p(x)=([x_1,\ldots,x_g]_3)a$ for some vector $a\in\mathbb{R}^{g^3}$, then
\begin{eqnarray*}
p^{\prime}(x)[h]&=&([h_1,\ldots,h_g]\otimes [x_1,\ldots,x_g]_2\\
& & +[x_1,\ldots,x_g]\otimes[h_1,\ldots,h_g]\otimes [x_1,\ldots,x_g] \\
& &+[x_1,\ldots,x_g]_2\otimes [h_1,\ldots,h_g])a\,.
\end{eqnarray*}
Higher order directional derivatives can be tracked in a similar fashion.

In calculations of this type the identity
\begin{equation}
\label{eq:Q3c}
([x_1, \dots, x_g]_{k}\,u)([y_1, \dots, y_{\ell}]v)=
([x_1, \dots, x_g]_k\otimes [y_1, \dots, y_{\ell}])\,(u\otimes v)\,,
\end{equation}
for noncommuting (as well as commuting )variables
$x_1, \dots, x_g$, $y_1, \dots, y_{\ell}$ and
vectors $u\in\mathbb{R}^{g^k}$ and $v\in\mathbb{R}^{\ell}$ is often useful.

\section {Basic Identities}
\label{sec:ident}

This section provides some basic identities needed for later sections.
Some are easy, some take a bit of work. Most of the details
are left to the
reader.

\subsection{Identities}

\begin{theorem}
\label{thm:ids}
Let $u\in\mathbb{R}^g$, $w\in\RR^{g^2}$, $A\in\RR^{g\times g}$,
$\varphi(x)=[x_1 \ \cdots \ x_g]u$
and let $\Pi_j$
denote the permutations defined in (\ref{eq:nov6d6}). Then the
following identities are in force:
\begin{enumerate}
\item $[x_1 \ \cdots \ x_g]_2 (u\otimes A)=
 \varphi (x)[x_1 \ \cdots \ x_g]A$.
\vspace{2mm}
\item $(u\otimes A)^T ([x_1 \ \cdots \ x_g]_2)^T
= A^T\begin{bmatrix} x_1 \\ \vdots \\ x_g \end{bmatrix} \varphi (x)$.
\vspace{2mm}
\item $[x_1 \ \cdots \ x_g](u\otimes A)^T = u^T \otimes (
[x_1 \ \cdots \ x_g]A^T)$.
\vspace{2mm}
\vspace{2mm}
\item $([x_1 \ \cdots \ x_g]_2)^T
= (I_g \otimes [x_1 \ \cdots \ x_g]^T) [x_1 \ \cdots \ x_g]^T$.
\vspace{2mm}
\item $\biggl([x_1 \ \cdots \ x_k] \begin{bmatrix} a_1 \\ \vdots \\
a_k \end{bmatrix} \biggl) \biggl([y_1 \ \cdots \ y_l]
\begin{bmatrix} b_1 \\ \vdots \\ b_l \end{bmatrix} \biggl)\\
= (b^T \otimes a^T) ([y_1 \ \cdots \ y_l]
\otimes [x_1 \ \cdots \ x_g])^T$.
\vspace{2mm}
\item
$\varphi^2 = (u^T \otimes u^T) ([x_1 \ \cdots \ x_g]_2)^T$, \\
$\varphi^3 = (u^T \otimes u^T \otimes u^T)
([x_1 \ \cdots \ x_g]_3)^T$, etc.
\vspace{2mm}
 \item $[x_1 \ \cdots \ x_g] uw^T ([x_1 \ \cdots \ x_g]_2)^T
 = (w^T \otimes u^T)([x_1 \ \cdots \ x_g]_3)^T$.
\vspace{2mm}
 \item $[x_1 \ \cdots \ x_g]_2 w u^T = u^T
 \otimes ((\Pi_1^T w)^T([x_1 \ \cdots \ x_g]_2)^T)$.
\vspace{2mm}
 \item $[x_1 \ \cdots \ x_g]_2 w u^T ([x_1 \ \cdots \ x_g])^T
  = (u^T \otimes (\Pi_1^T w)^T) ([x_1 \ \cdots \ x_g]_3)^T$.
\vspace{2mm}
 \item $\biggl( \begin{bmatrix} x_1 \\ \vdots \\ x_g \end{bmatrix}
 \otimes I_{g^k} \biggl) \begin{bmatrix} x_1 \\
  \vdots \\ x_g \end{bmatrix}_k = \begin{bmatrix} x_1 \\ \vdots
  \\ x_g \end{bmatrix}_{k+1}$ for $k=1,2,\ldots $.
\vspace{2mm}
\item If
 $w = \textup{vec}\, A$, then
\begin{eqnarray*}
[x_1 \ \cdots \ x_g]A &= &w^T
 \biggl(I_g \otimes \begin{bmatrix} x_1 \\ \vdots \\ x_g \end{bmatrix}
 \biggl)\quad\text{and}\\  ([x_1 \ \cdots \ x_g]_2)w&=& [x_1 \ \cdots \ x_g]
A^T([x_1 \ \cdots \ x_g])^T\,.
\end{eqnarray*}
\vspace{2mm}
 \item $A(u\otimes I_g)^T = u^T \otimes A$.
\vspace{2mm}
 \item $\biggl(I_{g^2} \otimes \begin{bmatrix} x_1 \\ \vdots \\ x_g
  \end{bmatrix}
 \biggl) ([x_1 \ \cdots \ x_g]_2)^T = ([x_1 \ \cdots \ x_g]_3)^T$.
\vspace{2mm}
\item If $v$ a vector of size $g^k$ and
$w$ a vector with  $g^l$ entries, then
$$
(v^T \otimes w^T) \begin{bmatrix} x_1 \\ \vdots \\ x_g
\end{bmatrix}_{k+l}
= \biggl(v^T \begin{bmatrix} x_1 \\ \vdots \\ x_g \end{bmatrix}_k \biggl)
\biggl(w^T \begin{bmatrix} x_1 \\ \vdots \\ x_g \end{bmatrix}_l \biggl)\,.
$$
\vspace{2mm}
\item If $a, \ c \in \mathbb{R}^k$ and  $ b, \ d \in \mathbb{R}^l$, then
$(a^T \otimes b^T)(c \otimes d) = (a^Tc) (b^Td)$.
\end{enumerate}
\end{theorem}

\begin{lemma}
\label{lem:Nicesplit}
If $u \in R^g$ and $w \in R^{g^2}$, then
\begin{eqnarray*}
[x_1 \ \cdots \ x_g]&uw^T& ([x_1 \ \cdots \ x_g]_2)^T
 + [x_1 \ \cdots \ x_g]_2 wu^T ([x_1 \ \cdots \ x_g])^T \\
&=& (w^T \otimes u^T + u^T \otimes (\Pi_1^T w)^T)
 ([x_1 \ \cdots \ x_g]_3)^T,
\end{eqnarray*}
where $\Pi_1$ is the permutation defined in (\ref{eq:nov6d6}).
\end{lemma}

\begin{proof}
This is immediate from identities (7) and (9) in Theorem \ref{thm:ids}.
\end{proof}

\subsection{Structured transposes}
\begin{lemma}
\label{lem:jun3a7}
If $u\in\RR^g$, $U=u\otimes I_g$ and $B\in\RR^{g\times g^2}$, then
\begin{equation}
\label{eq:jun3a7}
\left(BUU^T\right)^{sT}=uu^TB^{sT}\quad\text{and}\quad
\left(uu^TB\right)^{sT}=B^{sT}UU^T.
\end{equation}
\end{lemma}

\begin{proof}
Let $B=[b_{ij}]$, with $b_{ij}\in\RR^{1\times g}$ for $i,j=1,\ldots,g$.
Then the entries $c_{ij}$ of the matrix $C=BUU^T=B(uu^T\otimes I_g)$ are
\begin{eqnarray*}
c_{ij}&=&\sum_{s=1}^gb_{is}(u_su_jI_g)=\sum_{s=1}^gu_su_j b_{is}I_g\\
&=&\sum_{s=1}^gu_ju_sb_{is}
=\sum_{s=1}^g(uu^T)_{js}b_{is}\,,
\end{eqnarray*}
which coincides with the $ji$ entry of $uu^TB^{sT}$. Therefore,
$C^{sT}=uu^TB^{sT}$. This justifies the first formula in (\ref{eq:jun3a7}).
The second follows from the first upon replacing $B$ by $B^{sT}$, since
$\left( K^{sT}\right)^{sT}=K$ for every $K\in\RR^{g\times g^2}$.
\end{proof}

\begin{lemma}
\label{lem:aug12a8} If $u\in\RR^g$ and $A\in\RR^{g\times g}$, then
$$
\left(u^T\otimes A^T\right)^{sT}=u\left(\textup{vec} A\right)^T.
$$
\end{lemma}

\begin{proof}
If $A=[ a_1\cdots a_g]$ with $a_j\in\RR^g$ for $j=1,\ldots,g$, then the $ij$
entry of $u^T\otimes A^T$ is $u_ja_i^T$. Therefore, the $ij$ entry of
$(u^T\otimes A^T)^{sT}$ is $u_ia_j^T$, which coincides with the $ij$ entry
of $u(\textup{vec} A)^T$.
\end{proof}

\begin{lemma}
\label{lem:aug14c8}
If $u$, $v$, $w$ belong to $\RR^g$ and $y\in\RR^{g^2}$, then:
\begin{enumerate}
\item[\rm(1)]  $u(v^T\otimes w^T)=u\otimes (v^T\otimes w^T)$ and
$$
\left( u\otimes (v^T\otimes w^T)\right)^{sT}=u^T\otimes(v\otimes w^T)=
(v\otimes w^T)U^T=v(u^T\otimes w^T).
$$
\item[\rm(2)] $(uy^T)^{sT}=\left(\textup{mat}_g\,y\right)^TU^T$.
\end{enumerate}
\end{lemma}

\begin{proof} This is a straightforward computation that is similar to the
verification of the preceding two lemmas.
\end{proof}

\begin{lemma}
\label{lem:Z01sym}
The $g \times g^2$ matrix ${\cZ}_{01}$ is block symmetric in
the following sense:
$$
{\cZ}_{01}^{sT}={\cZ}_{01}.
$$
\end{lemma}

\begin{proof}
This rests on the formula
$$p_3 = \frac{1}{2} [x_1 \ \cdots \ x_g] {\cZ}_{01}
([x_1 \ \cdots \ x_g]_2)^T$$
(see (\ref{eq:nov10b6})) and the fact that  $p_3 = p_3^T$.
\end{proof}

\begin{rem}
\label{rem:oct16a8}
Lemma \ref{lem:Z01sym} is a special case of the more general fact that
if $p$ is a symmetric nc polynomial of degree $d$ and $j\le d-2$ and if the
$g\times g^{j+1}$ matrix $\cZ_{0j}$ is written in block form
as 
$$
\cZ_{0j}=\begin{bmatrix}b_{11}&\cdots &b_{1g}\\
\vdots & &\vdots \\
b_{g1}&\cdots & b_{gg}\end{bmatrix}
\quad \textup{with}\  b_{st}
\in\RR^{1\times g^j},
$$
then $b_{st}=b_{ts}$.
\end{rem}

\subsection{A factoring lemma}
The following lemma of
Vladimir Berkovich \cite{berk} simplifies a number of calculations.

\begin{lemma}
\label{lem:berk}
Let $\varphi(x)=\sum_{j=1}^ga_jx_j$ and $\psi(x)=\sum_{j=1}^gb_jx_j$.
Then the  identity
$$
\varphi(x)f_1(x)=f_2(x)\psi(x)
$$
for nc polynomials $f_1(x)=f_1(x_1,\ldots,x_g)$ and
$f_2(x)=f_2(x_1,\ldots,x_g)$ with $f_1(0)=f_2(0)=0$
implies that there exist a
polynomial $f_3(x)=f_3(x_1,\ldots,x_g)$ such that
$f_1(x)=f_3(x)\psi(x)$ and $f_2(x)=\varphi(x)f_3(x)$.
\end{lemma}

\begin{proof}
Without loss of generality we may assume that $b_g\ne 0$ and
introduce the nonsingular change of variables variables $y_1=x_1,
y_2=x_2,\ldots,y_{g-1}=x_{g-1}$
and $y_g=\sum_{j=1}^gb_jx_j$. Then the given equality is of the form
$$
(c_1y_1+\cdots + c_gy_g)\widetilde{f}_1(y)=\widetilde{f}_2(y)y_g
$$
and hence, upon writing $\widetilde{f}_1(y)=\sum_{i=1}^km_i$ as a sum of
monomials of degree at least one in the variables $y_1,\ldots,y_g$,
obtain that each monomial $m_i$ is
of the form $m_i=\widetilde{m_i}y_g$, where $\widetilde{m_i}$ is a monomial
that is one degree lower than $m_i$.
Therefore,
$\widetilde{f}_1(y)=\widetilde{f}_3(y)y_g$. The identity $f_1(x)=f_3(x)\psi(x)$  follows upon
rewriting  the last formula in terms of $x_1,\ldots,x_g$. The identity
$f_2(x)=\varphi(x)f_4(x)$ for some nc polynomial $f_4(x)$ is established in much the same way. But this in turn implies that $\varphi f_3\psi=\varphi f_4\psi$ and hence that $f_4=f_3$.
\end{proof}

\section{Classifying $\cZ$ with $\mu_-(\cZ) = 1$}
\label{sec:Zneg1}

In this section we study the structure of the symmetric  nc polynomial $p$
and the scalar middle matrix
$\cZ=Z(0)$ ($0\in\mathbb R^g$) associated with
its Hessian $p^{\prime\prime}(x)[h]$ when $\mu_-({\cZ})=1$. In view of
(\ref{eq:nov10a6}),
$$
\mu_-({\cZ})\le 1\Longrightarrow d\le 4 \Longrightarrow
\, \textrm{$\cZ$ is of the  form (\ref{eq:nov10c6})}\,.
$$
Moreover,
$$
\textup{rank}\,{\cZ}_{02}\le 1\,,
$$
thanks to the inequality
$$
\mu_{\pm}({\cZ})\ge \mu_{\pm}({\cZ}_{11})+\textup{rank}\,{\cZ}_{02}\,;
$$
see e.g., \cite{DHMjda} for a proof of the latter.

Thus, if
$\mu_-(\cZ)\le 1$, then
$$
p=p_0+p_1+p_2+p_3+p_4,
$$
where $p_j$ is either equal to a homogeneous polynomial of degree $j$ or
equal to zero. Moreover, the scalar middle matrices of
$p_4^{\prime\prime}(x)[h]$,
$p_3^{\prime\prime}(x)[h]$ and $p_2^{\prime\prime}(x)[h]$ are
$$
\begin{bmatrix}0&0&\cZ_{02}\\0&\cZ_{11}&0\\\cZ_{20}&0&0\end{bmatrix},\quad
\begin{bmatrix}0&\cZ_{01}\\ \cZ_{10}&0\end{bmatrix}\quad\text{and}
\quad\cZ_{00},
$$
respectively, and each of the polynomials $p_4$, $p_3$, $p_2$ can be
recovered from the scalar middle matrix of its Hessian by formula
(\ref{eq:nov10b6}).

 If $\mu_-(\cZ)\le 1$ and $p$ is not the zero polynomial, then there are
four mutually exclusive possibilities:
\begin{enumerate}
\item[\rm(1)] $\mu_-(\cZ)=1$ and $\textup{rank}\cZ_{02}=1$.
\vspace{2mm}
\item[\rm(2)] $\mu_-(\cZ)=1$, $\cZ_{02}=0$ and $\textup{rank}\cZ_{01}=1$.
\vspace{2mm}
\item[\rm(3)] $\mu_-(\cZ)=1$, $\cZ_{02}=0$, $\cZ_{01}=0$ and
$\mu_-(\cZ_{00})=1$.
\vspace{2mm}
\item[\rm(4)] $\mu_-(\cZ)=0$, $\cZ_{02}=0$, $\cZ_{01}=0$ and
$\mu_-(\cZ_{00})=0$.
\end{enumerate}

\subsection{The degree four case}
 \label{subsec:degreefour}
In this subsection we shall assume that
the rank of  ${\cZ}_{02}$ is
one
and hence that $p$ has degree four and
\begin{equation}
\label{eq:may9a7}
{\cZ}_{02} = uw^T\quad\textrm{with $u \in \mathbb{R}^g$,
$w \in \mathbb{R}^{g^3}$, $\Vert u \Vert = 1$ and $w\ne 0$}.
\end{equation}
Then $P=I_g -uu^T$
is the orthogonal projection of $\RR^g$ onto the orthogonal complement
of the vector $u$ in $\RR^g$.

To prove  Theorems \ref{thm:mainZ} and
\ref{thm:mainp}
it is convenient to first establish a number of lemmas.

\begin{lemma}
\label{lem:QZ}
If $\cZ$ is of the form (\ref{eq:nov10c6}) and if (\ref{eq:may9a7})
holds and
\begin{equation}
\label{eq:aug14a8}
E:=\begin{bmatrix} P{\cZ}_{00}P & P{\cZ}_{01} \\
 {\cZ}_{10}P & {\cZ}_{11}
\end{bmatrix},
\end{equation}
then:
\begin{enumerate}
\item[\rm(1)] $\mu_\pm(\cZ)=\mu_\pm(E)+1$.
\item[\rm(2)] $\mu_-(\cZ)=1\Longleftrightarrow E\succeq 0$.
\end{enumerate}
\end{lemma}

\begin{proof}
The identity
$$
G=\begin{bmatrix} P {\cZ}_{00}P & P{\cZ}_{01}& uw^T \\
{\cZ}_{10}P & {\cZ}_{11}&0\\
wu^T&0&0
\end{bmatrix}=K^T \cZ K
$$
with
$$
 K=\begin{bmatrix}
 I_g &0 & 0 \\ 0 & I_{g^2} & 0
 \\ -K_0 & -K_1 & I_{g^3}
 \end{bmatrix},
$$
$$
K_0=w(w^Tw)^{-1}u^T{\cZ}_{00} (I-\frac{1}{2}uu^T)
\quad\textrm{and}\quad K_1=w(w^Tw)^{-1}u^T{\cZ}_{01}
$$
implies that $\cZ$ is congruent to $G$ and hence that
$\mu_\pm(\cZ)=\mu_\pm(G)$. Moreover,
since
$$
G=G_1+G_2
$$
with
$$
G_1=\begin{bmatrix} P {\cZ}_{00}P & P{\cZ}_{01}& 0\\
{\cZ}_{10}P & {\cZ}_{11}&0\\
0&0&0
\end{bmatrix}\quad\text{and}\quad G_2=\begin{bmatrix}0&0& uw^T \\
0&0&0\\
wu^T&0&0
\end{bmatrix}
$$
and the ranges of these two real symmetric matrices are orthogonal,
it is readily checked
that
\begin{equation}
\label{eq:aug19a8}
\mu_\pm(G)=\mu_\pm(G_1)+\mu_\pm(G_2)=\mu_\pm(G_1)+1=\mu_\pm(E)+1,
\end{equation}
which justifies (1). Thus, as (2) is immediate from (1), the proof is
complete.
\end{proof}

The following lemma appears in \cite{DHMjda}. The statement
and proof are repeated here
for the convenience of the reader.

\begin{lemma}
\label{lem:p4rep}
If $\cZ$ is of the form (\ref{eq:nov10c6}) and (\ref{eq:may9a7}) holds, then
there exists a vector $w_1\in \mathbb R^{g^2}$ so that
\beq
\label{eq:nov8e6}
 w^T = u^T \otimes w_1^T.
\end{equation}
Moreover,
the homogeneous of degree four part of $p$ can be written as
\beq
\label{eq:nov8d6}
p_4(x) = \varphi(x) f_0(x) \varphi(x)
\end{equation}
where $\varphi(x)=[x_1,\ldots,x_g]u$ and $f_0$ is homogeneous of
degree two in $x$.

\end{lemma}

 \begin{proof}
 Recall that
\beq
  p_4 = \frac{1}{2} [x_1 \cdots x_g] \cZ_{02}([x_1 \cdots x_g]_3)^T \quad
\end{equation}
 and hence that
\beq
  p_4 = \frac{1}{2} [x_1 \cdots x_g] uw^T ([x_1 \cdots x_g]_3)^T.
\end{equation}
 Thus
$$
  p_4 = \varphi f\ \textrm{with}\ \varphi =[x_1 \cdots x_g]u\ \textrm{and}\
  f =\frac{1}{2}w^T([x_1 \cdots x_g]_3)^T\,.
$$
Since $p_4 = p_4^T$, it follows that $\varphi f = f^T \varphi,$ and hence
Lemma \ref{lem:berk} implies that
$$f = f_0 \varphi = \frac{1}{2}
[x_1 \cdots x_g]_2 w_0 u^T ([x_1 \cdots x_g])^T,$$
for some degree 2 homogeneous  polynomial
$$
f_0=\frac{1}{2}[x_1 \cdots  x_g]_2w_0\,.
$$
By identity (9) in Theorem \ref{thm:ids}, this can be written as
$$f = \frac{1}{2} (u^T \otimes (\Pi_1^T w_0)^T)([x_1 \cdots x_g]_3)^T\,,$$
where $\Pi_i$ is the permutation that is defined in formula (\ref{eq:nov6d6}).
Comparing the two formulas for $f$, yields the result
with $w_1=\Pi_1^T w_0$. \end{proof}

Recall that
$$ P=I_g-uu^T\quad\text{and}\quad U= u \otimes I_g.$$

\begin{lemma}
\label{lem:UAU}
If $\cZ$ is of the form (\ref{eq:nov10c6}) and (\ref{eq:may9a7}) holds, then:
\begin{enumerate}
\item[\rm(1)] $\cZ_{02}=u(u^T\otimes w_1^T)$ for some vector
$w_1\in\RR^{g^2}$.
\item[\rm(2)] $\cZ_{11} = UAU^T$, where $A=\textup{mat}_g(w_1)=A^T$.
\item[\rm(3)] ${\cZ}_{11} ={\cZ}_{11}UU^T = UU^T{\cZ}_{11}$.
\end{enumerate}
Moreover, if $\mu_-({\cZ})=1$, then
$$P{\cZ}_{01} = P{\cZ}_{01}UU^T$$
and there exist vectors $y\in \mathbb R^g$ and $v\in\mathbb R^{g^2}$
such that:
\begin{enumerate}
\item[\rm(4)] The entry $\cZ_{01}$ in formula (\ref{eq:jun3b7}) can be
expressed as
\begin{eqnarray*}
\cZ_{01}&=& u(u^T\otimes y^T)+ uv^T+(uv^T)^{sT}=\cZ_{01}^{sT}\\ \\
&=& u^T\otimes B^T+u(\textup{vec}B)^T=B^TU^T+u(\textup{vec}B)^T,
\end{eqnarray*}
where
$$
B=\textup{mat}_g\left(v+\frac{1}{2} u\otimes y\right)\quad\textrm{and}\quad
(\textup{mat}_gv)u=0.
$$
\item[\rm(5)] $PB^T=(\textup{mat}_gv)^T$.
\vspace{2mm}
\item[\rm(6)] If $E$ and $E_1$ denote the matrices in (\ref{eq:aug14a8}) and
(iv) of Theorem \ref{thm:mainZ}, respectively, and
$B$ is as in (4), then
$$
E_1=\begin{bmatrix}P\cZ_{00}P&PB^T\\BP&U^T\cZ_{11}U\end{bmatrix}=
\begin{bmatrix}P\cZ_{00}P&(\textup{mat}_gv)^T\\
(\textup{mat}_gv) &U^T\cZ_{11}U\end{bmatrix}
$$
and
$$
E\succeq 0\Longleftrightarrow E_1\succeq 0.
$$
\end{enumerate}
\end{lemma}

\begin{proof}
  Item (1) was established in Lemma \ref{lem:p4rep}. Item (2) rests on the
interplay between the formulas
\begin{equation*}
 \begin{split}
  p_4(x)= & \frac{1}{2} [x_1 \cdots x_g] {\cZ}_{02}([x_1 \cdots x_g]_3)^T \\
  p_4(x)= & \frac{1}{2} [x_1 \cdots x_g]_2 {\cZ}_{11}([x_1 \cdots x_g]_2)^T\,.
 \end{split}
\end{equation*}

In view of (1),
\begin{eqnarray*}
 p_4(x)
  &=& \frac{1}{2}\varphi(x)(u^T\otimes w_1^T)([x_1 \cdots x_g]_3)^T\\
&=& \frac{1}{2}\varphi(x)\left(([x_1 \cdots x_g]_3)(u\otimes w_1)\right)^T\\
&=& \frac{1}{2}\varphi(x)\left(\varphi(x)[x_1 \cdots x_g]_2w_1\right)^T\,.
\end{eqnarray*}
But, if the vector $w_1\in\RR^{g^2}$ is expressed as
$$
w_1=\textup{vec}\,A=\textup{vec}\,\begin{bmatrix}a_1&\cdots&a_g\end{bmatrix}
\quad\textrm{with}\quad a_i\in\RR^g\,,
$$
then
$$
[x_1 \cdots x_g]_2w_1=[x_1 \cdots x_g]A^T([x_1 \cdots x_g])^T
$$
and thus, as
$$
[x_1 \cdots x_g]u[x_1 \cdots x_g]=[x_1 \cdots x_g]_2U
$$
(by (1) of Theorem \ref{thm:ids} with $A=I_g$), the last formula for
$p_4(x)$ can be rewritten as
\begin{eqnarray*}
 p_4(x)
  &=& \frac{1}{2}\varphi(x)[x_1 \cdots x_g]A
\begin{bmatrix}x_1\\ \vdots\\x_g\end{bmatrix}\varphi(x)\\
&=&\frac{1}{2}[x_1 \cdots x_g]_2UAU^T([x_1 \cdots x_g]_2)^T\,,
\end{eqnarray*}
which, upon comparison with the formula for $p_4$ in terns of
$\cZ_{11}$, serves to justify (2).

Next, (3) is immediate from (2), since $U^TU=I_g$.

The rest of the proof is carried out under the added assumption that
$\mu_-({\cZ})=1$.
Then Lemma \ref{lem:QZ} implies that $E\succeq 0$ and hence, by a well
known argument (see e.g., Lemma 12.19 in \cite{dymbk}), that
$P{\cZ}_{01} = K {\cZ}_{11}$
for some $K\in \mathbb{R}^{g \times g^2}$.  Therefore,
$$
 P {\cZ}_{01}UU^T = K {\cZ}_{11} UU^T = K {\cZ}_{11}
 = P {\cZ}_{01},
$$
and, as $UU^T=uu^T\otimes I_g$,
\begin{equation}
 \label{eq:Z01}
  {\cZ}_{01} = P{\cZ}_{01}UU^T + uu^T {\cZ}_{01} (P\otimes I_g) +
uu^T {\cZ}_{01} (uu^T\otimes I_g)
         =\alpha+\beta+\gamma.
\end{equation}

The next step is to analyze the three terms
$$
\alpha=P{\cZ}_{01}UU^T,\quad \beta=uu^T {\cZ}_{01} (P\otimes I_g)\quad
\textrm{and}\quad \gamma= uu^T {\cZ}_{01} (uu^T\otimes I_g)
$$
in (\ref{eq:Z01}).

Let $y^T= u^T \cZ_{01}(u\otimes I_g)$. Then
$$
\gamma=uy^TU^T=u(u^T\otimes y^T),
$$
which has $ij$ block entry
$$
u_iu_jy^T\in\mathbb{R}^{1\times g^2}.
$$
Therefore, $\gamma^{sT}=\gamma$, since $u_iu_jy^T=u_ju_iy^T$.

Next, in view of Lemma \ref{lem:jun3a7}, the structured transpose
$\alpha^{sT}$ of $\alpha$ is
\begin{eqnarray*}
\alpha^{sT}&=&(\cZ_{01}UU^T)^{sT}-(uu^T\cZ_{01}UU^T)^{sT}\\
&=&uu^T{\cZ}_{01}^{sT}-uu^T(uu^T\cZ_{01})^{sT}\\
&=&uu^T{\cZ}_{01}^{sT}-uu^T\cZ_{01}^{sT}UU^T=uu^T\cZ_{01}(I_{g^2}-UU^T)\\
&=&uu^T{\cZ}_{01}^{sT}(P\otimes I_g)=\beta,
\end{eqnarray*}
since ${\cZ}_{01}^{sT}={\cZ}_{01}$ by Lemma \ref{lem:Z01sym}.

The first
advertised
form of ${\cZ}_{01}$ in (4)
is obtained by setting $v^T= u^T\cZ_{01}(P\otimes I_g)$, because then
$\beta= uv^T$ and
$\alpha=\beta^{sT}$. Moreover, since $v=(P\otimes I_g)w$ with $w=\cZ_{10}u$,
the identity
$$
\textup{mat}_g((P\otimes I_g)w)=(\textup{mat}_g w)P^T
$$
implies that $(\textup{mat}_g v)u=0$, since $P^Tu=Pu=0$. The second formula
for $\cZ_{01}$ in (4) follows from the first formula and
Lemma \ref{lem:aug12a8}.

Next, since $(\textup{mat}_g(u\otimes y))^T=uy^T$, it follows immediately
from the last formula for $B$ in (4) that $PB^T=P(\textup{mat}_g v)^T$.
But this yields (5), since $(\textup{mat}_g v)u=0$.

Finally, since $\cZ_{01}=B^TU^T+u(\textup{vec}\,B)^T$  by (4), and $Pu=0$,
$$
E=\begin{bmatrix}P\cZ_{00}P&P\cZ_{01}\\ \cZ_{10}P&\cZ_{11}\end{bmatrix}=
\begin{bmatrix}P\cZ_{00}P&PB^TU^T\\ UBP&\cZ_{11}\end{bmatrix}
$$
and hence, as $\cZ_{11}=UU^T\cZ_{11}UU^T$,
$$
E=\begin{bmatrix}I_g&0\\ 0&U\end{bmatrix}
\begin{bmatrix}P\cZ_{00}P&PB^T\\ BP&U^T\cZ_{11}U\end{bmatrix}
\begin{bmatrix}I_g&0\\ 0&U^T\end{bmatrix}=
\begin{bmatrix}I_g&0\\ 0&U\end{bmatrix}E_1
\begin{bmatrix}I_g&0\\ 0&U^T\end{bmatrix}.
$$
Thus, $E_1\succeq 0 \Longrightarrow E\succeq 0$. On the other hand,
the formula
$$
E_1=\begin{bmatrix}I_g&0\\ 0&U^T\end{bmatrix}E
\begin{bmatrix}I_g&0\\ 0&U\end{bmatrix}
$$
is also valid, since $U^TU=I_g$.
Therefore, $E\succeq 0 \Longrightarrow E_1\succeq 0$ and consequently
(6) holds.
\end{proof}

\begin{lemma}
\label{lem:jun12a7}
An nc polynomial ${\tt{p}}$ is of the form
  ${\tt{p}}=\varphi f \varphi$, where
\begin{itemize}
  \item[(i)] $\varphi =\sum u_j x_j$ and $u$ is a unit vector with entries
$u_1,\ldots,u_g$;
  \item[(ii)] $f=[x_1\ \cdots\ x_g]A [x_1\ \cdots\ x_g]^T$ with
$A=A^T\in\RR^{g\times g}$
\end{itemize}
 if and only if the
 scalar middle matrix $\mathfrak{Z}$ of ${\tt{p}}^{\prime\prime}$ is
\begin{equation}
\label{eq:may11a7}
  \mathfrak{Z} =2\begin{bmatrix} 0 & 0 & u (u^T\otimes w_A^T) \\
                                0 & UAU^T & 0\\
                              (u\otimes w_A)u^T & 0 & 0 \end{bmatrix},
\end{equation}
where $w_A=\textup{vec}(A)$ and $U=u\otimes I_g$.
\end{lemma}

\begin{proof}
Suppose first that
${\tt{p}}=\varphi f \varphi$
 where $\varphi$ and $f$ satisfy conditions (i) and (ii), respectively.
Then
\begin{equation*}
 \begin{split}
  {\tt{p}} =&   [x_1 \ \cdots \ x_g] u [x_1 \
  \cdots \ x_g]A([x_1 \ \cdots \ x_g])^T u^T([x_1 \ \cdots \ x_g])^T\\
  =&   [x_1 \ \cdots \ x_g]_2 UAU^T([x_1 \ \cdots \ x_g]_2)^T,
 \end{split}
\end{equation*}
where the first equality comes from the hypotheses and the
second from identity (1) (with $A=I_g$) in Theorem \ref{thm:ids}.
Hence $\mathfrak{Z}_{11}=2UAU^T$.

Similarly,
the formulas
\begin{eqnarray*}
\varphi(x)\begin{bmatrix}x_1 &\cdots &x_g\end{bmatrix}A
\begin{bmatrix}x_1 \\ \vdots \\ x_g\end{bmatrix}&=&
\varphi(x)\begin{bmatrix}x_1 &\cdots &x_g\end{bmatrix}_2w_A\\
&=&(\begin{bmatrix}x_1 &\cdots &x_g\end{bmatrix}_3)(u\otimes w_A)
\end{eqnarray*}
with $w_A=\mbox{vec}(A)$ imply that
$$
  {\tt{p}} = ([x_1
     \ \cdots \ x_g]_3)(u\otimes w_A)u^T( [x_1 \ \cdots \ x_g])^T,
$$
and hence that $\mathfrak{Z}_{20}=2(u\otimes w_A)u^T$.

Conversely, if $\mathfrak{Z}$ is of the form (\ref{eq:may11a7}), then the
formulas
$$
p=\frac{1}{2}[x_1,\ldots,x_g]_2\mathfrak{Z}_{11}([x_1,\ldots,x_g]_2)^T
$$
and (1) of Theorem \ref{thm:ids} (with $A=I_g$) lead easily to the
conclusion that
 ${\tt{p}}=\varphi f\varphi$, as claimed.
\end{proof}

\begin{lemma}
\label{lem:p3form}
An nc   polynomial ${\tt{p}}$ is of
 the form  ${\tt{p}}=\varphi q +q^T \varphi$, where
\begin{itemize}
 \item[(i)] $\varphi =\sum u_j x_j$ and $u$ is a unit vector with entries
$u_1,\ldots,u_g$; and
 \item[(ii)] $q=[x_1,\ldots,x_g]C\begin{bmatrix} x_1\\ \vdots\\x_g
\end{bmatrix}$ for some matrix $C\in\RR^{g\times g}$
\end{itemize}
if and only if  the scalar middle matrix $\mathfrak{Z}$ of
${\tt{p}}^{\prime\prime}$ is
\begin{equation}
\label{eq:jun3b7}
\mathfrak{Z}=
2\begin{bmatrix}
0&u^T\otimes C^T+u(\textup{vec}\,C)^T\\
u\otimes C+(\textup{vec}\,C)u^T&0\end{bmatrix}\,.
\end{equation}
\end{lemma}

\begin{proof}
If ${\tt{p}}=\varphi q +q^T \varphi$ with $\varphi$ and $q$ as described in
the first part of the lemma, then
$$
{\tt{p}}^{\prime\prime}= 2\varphi^\prime q^\prime+\varphi q^{\prime\prime}
+2 (q^\prime)^T\varphi^\prime+ (q^{\prime\prime})^T \varphi^\prime
$$
and the indicated form of $\mathfrak{Z}$ is easily obtained by direct
computation with the aid of the identities in Theorem \ref{thm:ids}.
Conversely, if $\mathfrak Z$ is of the form (\ref{eq:jun3b7}), then
direct computation based on the formula
$$
p=\frac{1}{2}[x_1,\ldots,x_g]{\cZ}_{01}([x_1,\ldots,x_g]_2)^T
$$
serves to recover ${\tt{p}}$.
\end{proof}

\subsection{The case of degree three}
 \label{subsec:degreethree}
In this subsection we assume that $\cZ_{02}=0$ and
that $\cZ_{01}$ is rank one.

.

\begin{lemma}
\label{lem:aug12b8}
If $\cZ_{i+j}=0$ in (\ref{eq:nov10c6}) for $i+j=2$ and $\cZ_{01}=uw_1^T$ for
some pair of
vectors $u\in\RR^g$ and $w_1\in\RR^{g^2}$ with $\Vert u\Vert=1$ and
$w_1\ne 0$
and $P=I_g-uu^T$, then:
$$
\mu_-(\cZ)=1\Longleftrightarrow P\cZ_{00}P\succeq 0.
$$
\end{lemma}

\begin{proof}
The formula
$$
\begin{bmatrix}\cZ_{00}&\cZ_{01}\\ \cZ_{10}&0\end{bmatrix}=
\begin{bmatrix}I_g&K^T\\0&I_{g^2}\end{bmatrix}
\begin{bmatrix}P\cZ_{00}P&uw_1^T\\ w_1u^T&0\end{bmatrix}
\begin{bmatrix}I_g&0\\K&I_{g^2}\end{bmatrix},
$$
with
$$
K=w_1(w_1^Tw_1)^{-1}u^T\cZ_{00}\left\{I_g-\frac{uu^T}{2}\right\}
$$
implies that
$$
\mu_-(\cZ)=\mu_-(P\cZ_{00}P)+\textup{rank}(uw_1^T)=\mu_-(P\cZ_{00}P)+1,
$$
which justifies the claim.
\end{proof}

\begin{lemma}
\label{lem:nov10a6}
An nc polynomial ${\tt{q}}$ of degree three is of the form
 ${\tt{q}}=\varphi \psi \varphi$, where
\begin{itemize}
 \item[(i)]  $\varphi =[x_1 \ \cdots \ x_g]u$, $u\in\RR^g$,
\vspace{2mm}
 \item[(ii)] $\psi=[x_1\ \cdots \ x_g]y$ and $y\in\RR^g$,
\end{itemize}
 if and only if the scalar middle matrix $\mathfrak{Z}$ of
$\tt{q}^{\prime\prime}$ is
 of the form
$$
 \mathfrak{Z} =2\begin{bmatrix} 0 & u(u^T\otimes y^T) \\
                (u \otimes y)u^T & 0 \end{bmatrix}.
$$
\end{lemma}

\begin{proof}
If ${\tt{q}}=\varphi \psi \varphi$, with $\varphi$ and $\psi$ as in (i)
and (ii), then
$$
{\tt{q}}^{\prime\prime}=2\varphi^\prime \psi^\prime\varphi+
2\varphi^\prime \psi\varphi^\prime+2\varphi \psi^\prime\varphi^\prime\,.
$$
However, the only term that contributes to $\gZ_{01}$ is
$$
2\varphi^\prime \psi^\prime\varphi=2[h_1\ \cdots\ h_g]u(\varphi
\psi^\prime)^T\,.
$$
Therefore, since
\begin{eqnarray*}
\varphi \psi^\prime&=&[x_1 \ \cdots \ x_g]u[h_1 \ \cdots \ h_g]y\\
&=&([x_1 \ \cdots \ x_g]\otimes [h_1 \ \ldots \ h_g])(u\otimes y)
\end{eqnarray*}
it is readily seen that
$$
\psi^\prime\varphi=(u^T\otimes y^T)([x_1 \ \cdots \ x_g]\otimes
[h_1 \ \cdots \ h_g])^T
$$
and hence that $\gZ_{01}=2u(u^T\otimes y^T)$. Conversely, if
$\gZ_{01}=2u(u^T\otimes y^T)$, then
the formula
$$
{\tt{q}}=\frac{1}{2}[x_1 \ \cdots \ x_g]\gZ_{01}([x_1 \ \cdots \ x_g]_2)^T
$$
implies that ${\tt{q}}=\varphi \psi \varphi$, as advertised.
\end{proof}

\begin{lemma}
\label{lem:nov10b6}
If ${\cZ}_{02}=0$ and ${\cZ}_{01}=u_1w_1^T$ with $u_1^Tu_1=1$,
$w_1\in\RR^{g^2}$ and
$w_1\ne 0$, then $w_1=u_1\otimes y_1$ for some nonzero vector $y_1\in\RR^g$,
\begin{equation}
\label{eq:nov13a6}
p_4(x)=0\quad \text{and}\quad p_3(x)=\varphi_1(x)f_1(x)\varphi_1(x)\,,
\end{equation}
where
$$
f_1(x)=\frac{1}{2}[x_1\ \cdots\ x_g]y_1\quad\text{and}\quad
\varphi_1(x)=[x_1 \ \cdots \ x_g]u_1\,.
$$
\end{lemma}

\begin{proof}
The formulas ${\cZ}_{01}=u_1w_1^T$  and (\ref{eq:nov10b6}) imply that
$$
p_3(x)=\frac{1}{2}[x_1\ \cdots \ x_g]\cZ_{01}([x_1\ \cdots \ x_g]_2)^T
=\varphi_1(x)f(x)^T,
$$
where
$$
f(x)=\frac{1}{2}([x_1\ \cdots \ x_g]_2)w_1.
$$
The rest follows from the fact that $p_3=p_3^T$ and Lemma \ref{lem:berk}.
\end{proof}

\subsection{Proof of Theorem \ref{thm:mainZ}}
\label{subsec:oct12a8}
If $\mu_-(\cZ)=1$ and $\textup{rank} \cZ_{02}=1$, then Lemmas
\ref{lem:QZ}--\ref{lem:UAU} guarantee that $\cZ$ is of the form
specified in part I of Theorem \ref{thm:mainZ} and that $E_1\succeq 0$.
 If $\mu_-(\cZ)=1$,
$\cZ_{02}=0$ and $\textup{rank} \cZ_{01}=1$, then Lemma \ref{lem:aug12b8}
guarantees that $\cZ$ is still of the form specified in part I of
Theorem \ref{thm:mainZ}, but
with $A=0$ and $v=0$ and that $P\cZ_{00}P\succeq 0$. Therefore,
$$
E_1=\begin{bmatrix}P\cZ_{00}P&0\\0&0\end{bmatrix}\succeq 0.
$$
If $\mu_-(\cZ)=1$, $\cZ_{02}=0$ and
$\cZ_{01}=0$, then there exists a unit vector $u\in\RR^g$ such that
$\cZ_{00}u=\lambda u$ with $\lambda <0$ and, if $P=I_g-uu^T$ for this
choice of $u$, then $P\cZ_{00}P\succeq 0$. If $\mu_-(\cZ)=0$, then
$\cZ_{02}=0$, $\cZ_{01}=0$ and $\cZ_{00}\succeq 0$.

Conversely, if $\cZ$ is of the form specified in Theorem \ref{thm:mainZ}, and
$E_1\succeq 0$, then, in view of Lemma \ref{lem:QZ} and (6) of
Lemma \ref{lem:UAU}, $\mu_-(\cZ)=1$ if $w_A\ne 0$. If $w_A=0$, then
$A=0$ and the constraint $E_1\succeq 0$ implies that $\textup{mat}_gv=0$
(see, e.g., Lemma 12.19 in\cite{dymbk}) and hence that $v=0$. Thus,
$\cZ_{01}=u(u^T\otimes y^T)$. If $y\ne 0$, then $\textup{rank}\cZ_{01}=1$
and $\mu_-(\cZ)=1$. If $y=0$, then $\cZ_{01}=0$ and
$\mu_-(\cZ)=\mu_-(\cZ)\le 1$, since $P\cZ_{00}P\succeq 0$.

\subsection{Proof of Theorem \ref{thm:mainp}}
\label{subsec:oct12b8}
Since $\mu_-({\cZ})=\sigma_-^{min}(p^{\prime\prime})$, the assumption
$\sigma_-^{min}(p^{\prime\prime})\le 1$ guarantees that the scalar middle
matrix of the Hessian $p^{\prime\prime}(x)[h]$ of $p$ has the form
indicated in Theorem $\ref{thm:mainZ}$ and that $E_1\succeq 0$.
Consequently, the homogeneous
components of $p$ of degree $j$ with $j\ge 2$ may be computed from the
entries in $\cZ$ and formula (\ref{eq:nov10b6}). Moreover, by (1) of
Theorem \ref{thm:ids},
\begin{eqnarray*}
p_4(x)&=&\frac{1}{2}[x_1\ \cdots \ x_g]_2\cZ_{11}([x_1\ \cdots \ x_g]_2)^T\\
&=&\frac{1}{2}[x_1\ \cdots \ x_g]_2UAU^T([x_1\ \cdots \ x_g]_2)^T\\
&=&\frac{1}{2}\varphi(x)[x_1\ \cdots \ x_g]A([x_1\ \cdots \ x_g])^T\varphi(x),
\end{eqnarray*}
which implies that $A=2Q(f_0)$. Similarly the formula
$$
p_3(x)=\frac{1}{2}[x_1\ \cdots \ x_g]\cZ_{01}([x_1\ \cdots \ x_g]_2)^T
$$
together with the formulas for $\cZ_{01}$ in (4) of Lemma 4.3 imply that
$$
\varphi(x)q(x)=
\frac{1}{2}[x_1\ \cdots \ x_g]u(\textup{vec} B)^T([x_1\ \cdots \ x_g]_2)^T
$$
and hence (with the help of the transpose of the second formula in (11)
of Theorem \ref{thm:ids}) that
$$
q(x)=\frac{1}{2}(\textup{vec} B)^T([x_1\ \cdots \ x_g]_2)^T
=\frac{1}{2}[x_1\ \cdots \ x_g]B([x_1\ \cdots \ x_g])^T.
$$
Therefore, $B=2Q(q)$. Similarly, the formula
$$
p_2=\frac{1}{2}[x_1\ \cdots \ x_g]\cZ_{00}([x_1\ \cdots \ x_g])^T
$$
implies that $\cZ_{00}=2Q(p_2)$. Thus, in view of (6) of Lemma \ref{lem:UAU},
the matrices $E_2$ in
Theorem \ref{thm:mainp}  and $E_1$ in Theorem \ref{thm:mainZ}
are simply related: $E_1=2E_2$.

If $p$ is of degree three, then $v=0$ by Theorem \ref{thm:mainZ} and hence
the formula for $B$ in (4) of Lemma \ref{lem:UAU} reduces to
$2B=\textup{mat}_g(u\otimes y)=yu^T$. Therefore,
$$
q(x)=\frac{1}{4}[x_1\ \cdots \ x_g]yu^T([x_1\ \cdots \ x_g])^T=f_1(x)\varphi(x)
$$

Conversely, if $p$ is the
form specified in Theorem \ref{thm:mainp}, then Lemmas \ref{lem:jun12a7},
\ref{lem:p3form} and \ref{lem:nov10a6}
 guarantee that
$\cZ$ is of the form specified in Theorem \ref{thm:mainZ} and hence that
$\mu_-({\cZ})=\sigma_-^{min}(p^{\prime\prime})\le 1$.

\subsection{Proof of Theorem \ref{thm:mainZsp}}
\label{subsec:oct12c8}
Theorem \ref{thm:mainZsp} follows from the formulas

$p_4(x)=\frac{1}{2}[x_1\cdots x_g]_2\cZ_{11}([x_1\cdots x_g]_2)^T$,
$p_3(x)=\frac{1}{2}[x_1\cdots x_g]\cZ_{01}([x_1\cdots x_g]_2)^T$ and
$p_2(x)=\frac{1}{2}[x_1\cdots x_g]_2\cZ_{00}([x_1\cdots x_g])^T$,
and appropriate choices of the identities in Theorem \ref{thm:ids}.


\section{Modified Hessians}
\label{sec:iMHintro}

In this section,  we introduce the modified Hessian
$$p^{\prime\prime}_\lambda = p^{\prime\prime}_\lambda(x, h)=
p^{\prime\prime}(x)[h] + \lambda\{p'(x)[h]\}^T \{p'(x) [h]\}$$ of
a symmetric polynomial $p(x) = p(x_1, \dots, x_g)$ in $g$
noncommuting variables. The first order of business is to
establish a representation formula analogous to formula
(\ref{eq:defs-middle}) for the new term $\{ p'(x) [h]\}^T \{p'(x) [h]\}$.

\begin{lemma}
\label{lem:nov6a6}
Let $p_k(x)$ be a symmetric nc polynomial that is homogeneous of degree
$k$ in
the $g$ symmetric variables $x_1,\ldots,x_g$ and suppose that $k\ge 1$. Then
$p_k^{\prime}(x)[h]$ can be expressed uniquely in the form
\begin{equation}
\label{eq:nov6a6}
p_k^{\prime}(x)[h]=\sum_{j=0}^{k-1}\varphi_{kj}(x)V_j(x)[h]\,,
\end{equation}
where $\varphi_{kj}(x)$ is a row polynomial of size $1\times g^{j+1}$
in which the nonzero entries are homogeneous polynomials of degree $k-1-j$.
\end{lemma}

\begin{proof}
The polynomial $p_k(x)$ can be expressed in the form
$$
  p_k(x)=u_k^T\begin{bmatrix}x_1\\ \vdots\\x_g\end{bmatrix}_k
$$
 for some vector
$u_k\in\mathbb{R}^{g^k}$. Therefore,
\begin{eqnarray*}
p_k^{\prime}(x)[h]&=&u_k^T\left\{\begin{bmatrix}h_1\\ \vdots\\h_g\end{bmatrix}
\otimes\begin{bmatrix}x_1\\ \vdots\\x_g\end{bmatrix}_{k-1}+
\begin{bmatrix}x_1\\ \vdots\\x_g\end{bmatrix}\otimes
\begin{bmatrix}h_1\\ \vdots\\h_g\end{bmatrix}\otimes
\begin{bmatrix}x_1\\ \vdots\\x_g\end{bmatrix}_{k-2}+\cdots\right.\\
&{}&\left.\qquad\qquad +
\begin{bmatrix}x_1\\ \vdots\\x_g\end{bmatrix}_{k-1}\otimes
\begin{bmatrix}h_1\\ \vdots\\h_g\end{bmatrix}\right\}\\
&=&u_k^T\left\{\Pi_{k-1}V_{k-1}+\sum_{i=1}^{k-1}
\begin{bmatrix}x_1\\ \vdots\\x_g\end{bmatrix}_i
\otimes\Pi_{k-1-i}V_{k-1-i}\right\},
\end{eqnarray*}
where the $\Pi_j$ denote the permutations defined by formula (\ref{eq:vja}) for
$j=1,\ldots,k-1$ and $\Pi_0=I_g$. But the last formula for
$p_k^{\prime}(x)[h]$
can be rewritten in the form (\ref{eq:nov6a6}) by noting that
$$
u_k^T\left\{\begin{bmatrix}x_1\\ \vdots\\x_g\end{bmatrix}_i\otimes
\Pi_{k-1-i}V_{k-1-i}\right\}=
[x_1,\ldots,x_g]_iA_i\Pi_{k-1-i}V_{k-1-i}
$$
for a suitably defined matrix $A_i\in\mathbb{R}^{g^i\times g^{k-i}}$ and then
setting
$$
\varphi_{k,k-1-i}(x)=\left\{\begin{array}{ll}u_k^T
[x_1,\ldots,x_g]_iA_i\Pi_{k-1-i}&\quad\text{for}\ i=1,\ldots,k-1\\ \\
u_k^T\Pi_{k-1}&\quad\text{for}\ i=0.\end{array}\right.
$$
\end{proof}

\begin{lemma}\label{lem:3}
If $p(x)$ is a symmetric nc polynomial
 of degree $d$ in $g$ symmetric
variables,
    then
\begin{eqnarray}\label{eq:rep2}
\lefteqn{\{p'(x) [h]\}^T\{p'(x)[h]\}=}\nonumber\\
&&[V^T_0, V^T_1,  \ldots,  V^T_k]\ \left[
\begin{array}{cccc}
W_{00}& W_{01} & \cdots & W_{0k}\\
W_{10}& W_{11} & \cdots & W_{1k}\\
\vdots& \vdots & &\vdots\\
W_{k0}& W_{k1} & \cdots & W_{kk}
\end{array}
\right]\ \left[\begin{array}{c} V_0\\ \vdots\\ V_k
\end{array}\right],
\end{eqnarray}
where:
\begin{enumerate}
\item[\rm(1)]  $k = d -1$.

\item[\rm(2)]  The vectors $V_j=V_j(x)[h]$ in formula (\ref{eq:rep2})
are given by formula (\ref{eq:vjt}) for $j=0, \dots, d-1$.

\item[\rm(3)]  $W_{ij}$ is a matrix of size $g^{i+1} \times g^{j+1}$
and the entries in $W_{ij}$ are polynomials in the noncommuting
variables $x_1, \dots, x_g$ of degree $\le 2(d-1) - (i+j)$.

\item[\rm(4)] $W^T_{ij} = W_{ji}$.

\item[\rm(5)] The matrix $W=W(x)$ admits the factorization
$$
\left[\begin{array}{ccc}
W_{00} & \cdots& W_{0k}\\
\vdots&&\vdots\\
W_{k0} &\cdots&W_{kk}\end{array}\right] =
\left[\begin{array}{c} \psi_0\\ \vdots\\ \psi_k
\end{array}\right]
[\psi^T_0 \cdots \psi^T_k].
$$
\end{enumerate}
\end{lemma}

\begin{proof}
Items (1)--(4) are straightforward; (5) is discussed next.

If $p(x)$ is a polynomial of degree $d$ in $g$ noncommuting
variables, then
$p=\sum_{j=0}^dp_j(x)$, where $p_j$ is either equal to a homogeneous polynomial of degree $j$ or to zero and $p_d$ is not zero. Therefore,
\begin{equation}
\label{eq:N1}
p^{\prime}(x)[h]=\sum_{j=1}^dp^{\prime}_j(x)[h]=
\sum\limits_{s=0}^{d-1}\psi_s(x)^T\;V_s(x)[h]\,,
\end{equation}
follows by applying Lemma \ref{lem:nov6a6}
to each of the terms $p_j(x)$ in the sum.
Thus,
$p^{\prime}(x)=p^{\prime}(x)[h]$
can be expressed in terms of the border vectors $V_s=V_s(x)[h]$
and a unique choice of vector polynomials $\psi_s=\psi_s(x)$
of size $g^{s+1}\times 1$ and degree $d-1-s$ by the indicated formula.
Consequently, the entries $W_{ij}$ in the representation formula
(\ref{eq:rep2}) can now be written in terms of the vector polynomials
$\psi_0,\ldots,\psi_k$ as
\begin{equation}
\label{eq:N1a}
W_{ij}(x)=\psi_i(x)\psi_j^T(x)\quad\textrm{for}\quad i,j\le d-1
\end{equation}
and hence the full matrix
\begin{equation}
\label{eq:N1b} W(x)=\left[\begin{array}{c}\psi_0(x)\\ \vdots \\
\psi_k(x)\end{array}
\right][\psi_0^T(x),\ldots,\psi_k^T(x)]\,,\quad\textrm{where}\quad
k=d-1\,.
\end{equation}
\end{proof}

 The {\bf middle matrix for the modified Hessian} of a symmetric  polynomial
 of degree $d$ is the $(d-1)\times (d-1)$ block matrix $Z_\lambda$ (with
  polynomial entries)
$$
 Z_\lambda = \begin{bmatrix} Z & 0\\ 0 & 0\end{bmatrix} + \lambda W,
$$
 where $Z$ is the middle matrix for $p^{\prime\prime}$.
 Thus, as $\widetilde{V}(x)[h]$
 denotes the border vector of height $g\widetilde{\nu}$ (which includes
monomials in
 $x$ of degree $d-1$), the modified Hessian can be represented as
$$
  p_{\lambda}^{\prime\prime}(x)[h]= \widetilde{V}(x)[h]^T  Z_\lambda(x)
\widetilde{V}(x)[h].
$$
  The {\bf scalar middle matrix of the modified Hessian} is the
  matrix
\begin{equation}
 \label{eq:defscalarmidmod}
  \cZ_\lambda =
    \begin{bmatrix} \cZ & 0\\ 0 & \lambda\psi_{d-1}(0)\psi_{d-1}(0)^T
\end{bmatrix}.
\end{equation}
 Note that this differs a bit from the earlier terminology since
 $Z_\lambda(0)\ne \cZ_\lambda$. The key point is that
  $\cZ_\lambda$ is constant and is polynomially congruent
 to $Z_\lambda(x)$ via a congruence which does not depend upon $\lambda$.

\begin{theorem}
\label{thm:nov7a6}
Let $p(x)$ be a symmetric nc polynomial of degree $d\ge 1$ in $g$
symmetric  variables
and let $\psi_j(x)$ denote the coefficients of $p'(x)[h]$
in formula {\rm (\ref{eq:N1})}.
Then
$$
\psi_{d-1}(x)=\psi_{d-1}(0)\ne 0
$$
and there exists a block $d\times d$ matrix-valued  polynomial
$S$ with polynomial inverse so that
$$
  Z_\lambda(x)=S(x)^T \cZ_\lambda S(x).
$$
In particular,  $S(X)$ is invertible when $X\in \gtupn$, and
$$
 Z_\lambda(X) = S(X)^T (\cZ_\lambda \otimes I_n) S(X).
$$
\end{theorem}

\begin{proof}
It is convenient to let $y=\psi_{d-1}(x)$ and
$\tpsi^T :=[\psi_0^T,\ldots,\psi_{d-2}^T]$.
Then, as $y\in\RR^{g^d}$,
\begin{eqnarray}
\\
Z_\lambda(x) &=&
\begin{bmatrix} Z(x) & 0 \\ 0 & 0 \end{bmatrix}
+ \lambda \begin{bmatrix} \tpsi \tpsi^T & \tpsi y^T \\
y \tpsi^T & yy^T
\end{bmatrix}
=\begin{bmatrix} Z(x) + \lambda \tpsi \tpsi^T &
 \lambda \tpsi y^T \\
\lambda y \tpsi^T & \lambda yy^T
\end{bmatrix} \nonumber\\
&=&
\begin{bmatrix} I & \lambda \tpsi y^T (\lambda
yy^T)^{\dagger} \\ 0 & I \end{bmatrix}
\begin{bmatrix} Z(x) & 0 \\ 0 & \lambda y y^T \end{bmatrix}
\begin{bmatrix} I & 0 \\
(\lambda yy^T)^{\dagger} \lambda y \tpsi^T
 & I \end{bmatrix}\nonumber\\
&=&
\begin{bmatrix} I &  \tpsi y^T (
yy^T)^{\dagger} \\ 0 & I \end{bmatrix}
\begin{bmatrix} Z(x) & 0 \\ 0 &  \lambda y y^T \end{bmatrix}
\begin{bmatrix} I & 0 \\
( yy^T)^{\dagger}  y \tpsi^T
 & I \end{bmatrix}\,,\nonumber
\end{eqnarray}
since the Moore-Penrose inverse $(\lambda yy^T)^{\dagger}$
of $\lambda yy^T$ is given by the formula
$$(\lambda y y^T)^{\dagger}
=
 \frac{y(y^T y)^{-2} y^T}{\lambda}$$
when $\lambda\ne 0$, and
$$y^T (\lambda y y^T)^{\dagger} (\lambda y y^T)
 =
 y^T (\lambda y y^T)(\lambda y y^T)^{\dagger}
 =
 y^T\,.$$
Thus, upon setting $C(x)=(yy^T)^\dagger y\tpsi^T(x)$ and invoking formula
(\ref{eq:may13a7}), it follows that
$$
{Z}_\lambda(x)=S(x)^T{\cZ}_\lambda S(x)\,,
$$
where
\begin{equation}
\label{eq:may13b7}
  S(x)=\begin{bmatrix} B(x) & 0\\ C(x) & I \end{bmatrix}\quad\textrm{and}\quad
S(x)^{-1}=\begin{bmatrix} B(x)^{-1} & 0\\ -C(x)B(x)^{-1} & I \end{bmatrix}
\end{equation}
are both polynomial matrices.
It is important to note that $S(x)$
does not depend on $\lambda$.
\end{proof}

\begin{cor}
If $\lambda>0$, then, in the setting of Theorem \ref{thm:nov7a6},
\begin{enumerate}
\item[\rm(1)] $\mu_+({\cZ}_{\lambda})=\mu_+({\cZ})+1$ and
$\mu_-({\cZ}_{\lambda})=\mu_-({\cZ})$.
\item[\rm(2)] ${\cZ}_{\lambda}\succeq 0\Longleftrightarrow {\cZ}\succeq 0$.
\item[\rm(3)] ${\cZ}_{\lambda}\succeq 0\Longrightarrow d\le 2$ (thanks to
(\ref{eq:nov10a6}).
\end{enumerate}
\end{cor}

\section{The Relaxed Hessian: Local Positivity}
 \label{sec:relax}
 In this section we exploit the structure of the
 polynomial congruence of Theorem \ref{thm:absModHess}
 to prove Theorems \ref{thm:local} and  \ref{thm:localq}.
 We then  establish
 Theorems \ref{thm:regHess-k} and \ref{thm:regHessq-k}.

 \subsection{Proof of Theorem \ref{thm:local}}
\label{sec:pfabsModHess}
 The starting point is the polynomial congruence formula (\ref{eq:may13b7})
of Theorem \ref{thm:nov7a6}.
Again, we emphasize that, whereas $\cZ=Z(0)$ for $0\in \mathbb R^g$,
 it is not the case that  $\cZ_\lambda = Z_\lambda(0)$.

It is convenient to let $F=(S^T)^{-1}S^{-1}$, and to bear in mind that
$\cZ_\lambda(X)=\cZ_\lambda\otimes I_n$ for $X\in\gtupn$,
since $\cZ_\lambda$ is constant.

If $X\in\gtupn$ and  $\delta>0$, then
$$
Z_\lambda(X) +\delta I
   = S(X)^T (\cZ\otimes I_n+\delta F(X)) S(X)\,.
$$
 Thus, for $H\in\gtupn$ and $v\in\mathbb R^n$, we have
\begin{equation}
\label{eq:local-Y}
  \left\langle \left
    ( p^{\prime\prime}_\lambda(X)[H] +\delta \tV(X)[H]^T \tV(X)[H]\right ) v,v
    \right\rangle
\end{equation}
$$
   = \langle ( ( \cZ_\lambda\otimes I_n
   + \delta F(X)) S(X)\tV(X)[H] ) v, S(X) \tV(X)[H]v \rangle.
$$

 The hypothesis that $\{\tV(X)[H]v: H\in\gtupn \}$ has codimension
 at most $n-1$ in $\mathbb R^{ng\widetilde{\nu}}$ and the invertibility of
$S(X)$
 implies  that $\{S(X)\tV(X)[H]v: H\in\gtupn\}$ also has codimension
 at most $n-1$.
 Consequently, as the relaxed Hessian is positive at $(X, v)$
under the hypothesis of Theorem \ref{thm:local},
 we can select $\lambda$ such that the left side of  formula
(\ref{eq:local-Y})
 is nonnegative for each $H\in\gtupn$, in which case
 $\cZ_\lambda\otimes I_n  +\delta F(X)$ has at
 most $n-1$ negative eigenvalues.
It follows that $\cZ \otimes I_n + \delta {\widetilde F}(X),$ the
upper left hand (block) corner of
$\cZ_\lambda\otimes I_n   + \delta F(X),$
has at most $n-1$ negative eigenvalues for each $\delta>0$.
Since this upper left hand corner does not depend upon $\lambda$
and $\delta>0$ is arbitrary,
 it follows that $\cZ \otimes I_n$ has at most $n-1$ negative eigenvalues.
 Therefore, by  (\ref{eq:apr30a7}),
  $\mu_-(\cZ)=0$. Thus, $p$ has degree at most two
  by (\ref{eq:nov10a6}).

\subsection{Proof of Theorem \ref{thm:localq}}
\label{sec:pflocalq}
 To see that Theorem \ref{thm:localq} follows from
 Theorem \ref{thm:local} it
suffices to show that, under the hypothesis on $(X,v)$
 in Theorem \ref{thm:localq},  the
 subspace $\{\tV(X)[H]v: H\in \gtupn\}$  has codimension
 at most $n-1$ in $\mathbb R^{ng\widetilde{\nu}}.$
 For this, and subsequent proofs, we will invoke an
 estimate  that is extracted from Lemmas 9.5 and 9.7 in \cite{CHSY03},
which are rephrased below as Lemma
\ref{lem:CHSY} for the convenience of the reader.

\begin{lemma}[CHSY Lemma]
 \label{lem:CHSY}
  Given a pair of positive integers $g$ and $r$, a matrix $X\in\gtupn$
  and a vector $v\in \mathbb R^{n}$, let
$$
\alpha_k=\sum_{j=0}^kg^j\,, \quad
\mathcal{R}_k=\left\{\begin{bmatrix}V_0(X)[H]v\\
\vdots\\V_k(X)[H]v\end{bmatrix}:\, H\in\gtupn\right\}\,,
$$
and suppose that the set $\{m(X)v: |m|\le r\}$ is a linearly independent
subset
  of $\mathbb R^n$. Then $\mathcal{R}_k$ is a subspace of $\RR^{ng\alpha_k}$
and its codimension
\begin{equation}
\label{eq:aug22a7}
\textup{codim}\,\mathcal{R}_k\le ng(\alpha_k-\alpha_r)+g\alpha_r
\frac{\alpha_r-1}{2}
\quad\text{if}\quad k\ge r\,,
\end{equation}
with equality if $k=r$. (Recall that $\vert m\vert$ denotes the length of the
monomial $m$.)
\end{lemma}
A key fact is  that if $k=r$, then the bound on the codimension is
independent of $n$.

The bound (\ref{eq:aug22a7}) follows easily from the following sequence of
estimates, the first of which is based on Lemmas 9.5 and 9.7 in [CHYS03].
\begin{enumerate}
\item[\rm(1)] The codimension of ${\cR}_r$ in
$\RR^{ng\alpha_r}$ is equal to $g\alpha_r(\alpha_r-1)/2$.
\item[\rm(2)] $\textup{dim}\,{\cR}_r=ng\alpha_r-g\alpha_r(\alpha_r-1)/2$.
\item[\rm(3)] $\textup{dim}\,{\cR}_k\ge ng\alpha_r-g\alpha_r(\alpha_r-1)/2$
if
$k\ge r$.
\end{enumerate}
  To prove Theorem \ref{thm:localq}, note that
  the set $\{m(X)v:|m|\le d-1\}$ is linearly dependent
  if and only if there exists a set of nonzero
numbers $q_m\in\RR$ for $|m|\le d-1$
such that
$$
  q(X)v= \sum q_m m(X)v = 0.
$$
  Thus, $\{m(X)v:|m|\le d-1\}$ is linearly independent.
  By the CHSY Lemma, with $r=d-1$ so that $\alpha_r=\widetilde{\nu}$,
the codimension of the set $\cR_{d-1}=\{\tV(X)[H]v: H\in\gtupn\}$
  is less than $n$ (thanks to assumption (\ref{eq:may13c7})). The proof is
  completed by an application of Theorem \ref{thm:local}.

\subsection{Proof of Corollary \ref{cor:tomainModHess}}
\label{subsec:oct16a8}
We begin with a definition and a preliminary lemma.

If $\mathfrak{B}_n\subset (\RR_{sym}^{n\times n})^g\times \RR^n$ for
$n=1, 2,\dots$, then the graded set
$$
\mathfrak B=\cup_{n\ge 1}\mathfrak B_n
$$
is {\bf closed with respect
 to direct sums} if whenever $(X^j,v^j)\in\mathfrak B_{n_j},$
 for $j=1,2,\dots,k,$ it follows that $(X, v)$ is in
$\mathfrak B_n$ for $n=\sum n_j$, where
$X=\textup{diag\{}X^1,\ldots,X^k\}$ and $v=\textup{vec}[v^1\ \cdots\ v^k]$.

\begin{lemma}
 \label{lem:directsums}
  Let $k$ be a given positive integer
and suppose $\mathfrak B$ is closed with respect to
 direct sums. If for each $(X,v)\in\mathfrak B$ the set
 $\{m(X)v: |m|\le k\}$ is linearly dependent, then there
 exists a nonzero  polynomial
$$
   q=\sum_{|m|\le k}  q_m m(x)
$$
(which is not necessarily symmetric) of degree $\le k$ such that
$q(X)v=0$ for every $(X,v)\in\mathfrak B$.
\end{lemma}

\begin{proof} See Lemma 4.1 of \cite{DHMind} and, for a cleaner proof,
\S 6 of \cite{HMVjfa}.
\end{proof}

To prove Corollary \ref{cor:tomainModHess}, note that there does not
exist a nonzero polynomial $q$ of degree at most $k$ such that
$q(X)v=0$ for every $X\in\gtupn$ with  $\|X\|<\epsilon$ when $v$ is a
nonzero vector in $\RR^n$. Thus, Lemma \ref{lem:directsums} applied to
the noncommutative $\varepsilon$ neighborhood of zero (i.e., to the graded
set $\mathcal{U}$  that is defined just above the statement of Corollary
\ref{cor:tomainModHess}) guarantees that there exists a pair $(X,v)\in
\mathcal{U}_n$
for some choice
of $n$ such that $\{m(X)v:\vert m\vert\le k\}$ is linearly independent in
$\RR^n$. This much is true for every positive integer $k$. Now fix
$k=d-1$. Then $n\ge \alpha_{d-1}=\widetilde{\nu}$.  However, by
considering $\oplus_1^t (X,v)$
it may be assumed that $n>g\widetilde{\nu}(\widetilde{\nu}-1)/2$, and
hence that Theorem \ref{thm:localq} is applicable.

\subsection{Proofs of Theorems \ref{thm:regHess-k} and \ref{thm:regHessq-k}}
\label{sec:pfregHess-k}
 For the proof of
 Theorem \ref{thm:regHess-k}, the now familiar arguments show that the
hypotheses imply
 that $\mu_-(Z(X)) \le kn-1.$ On the other hand,
 $\mu_-(Z(X))=n \mu_-(\cZ)$. Hence
 $\mu_-(\cZ)=\sigma_-^{min}(p^{\prime\prime})< k.$

 Theorem \ref{thm:regHessq-k} follows from Theorem
 \ref{thm:regHess-k} in much the same way that
  Theorem \ref{thm:localq} follows from Theorem \ref{thm:local}.
  The main point is that, from the CHSY-Lemma with $r=d-2$ so that
$\alpha_r=\nu$, the subspace $\cR_{d-2}$ has
  codimension less than $n$ (thanks to assumption (\ref{eq:aug20a7})).
  Thus, restricting $H$ to the space $\mathcal H$ of codimension at most $nk$,
  it follows that
\begin{eqnarray*}
\textup{codim}\{V(X)[H]v: H\in\mathcal H\}&\le&
\textup{codim}\,\cR_{d-2}+\textup{codim}\,\cH\\
&\le& n-1+nk.
\end{eqnarray*}


\appendix

\section{A connection between $Z$ and $W$} 
\label{sec:9}

In this appendix we present some useful connections between the polynomials
$\psi_j(x)$ that are defined by formula (\ref{eq:N1}) and the entries in the
middle matrix $Z_{ij}(x)$ for the Hessian $p^{\prime\prime}(x)[h]$ for an
arbitrary symmetric nc polynomial of degree $d$ in $g$ symmetric variables.

\begin{theorem}
\label{thm:N1}
If $p(x)$ is a symmetric nc polynomial of degree $d$ in $g$ symmetric
variables,
then the coefficients $\psi_s$, $s=0, \ldots, d-1$,
in formula (\ref{eq:N1})
are related to the entries
$Z_{0s}$  in the representation formula (\ref{eq:defs-middle}) for
$p^{\prime\prime}(x)$ by
the formula
\begin{equation}
\label{eq:N2}
\psi_s(x)^T=\frac{1}{2}[x_1 \ \cdots \ x_g]Z_{0s}(x) \ + \ \psi_s(0)^T
\quad\textrm{for}
\quad s=0, \ldots, d-1\,.
\end{equation}
\end{theorem}

\begin{proof}
In view of formulas (\ref{eq:nov10b6}) and (\ref{eq:nov6d6})
\begin{eqnarray*}
p(x)&=&\frac{1}{2}\sum_{j=0}^\ell [x_1\ \cdots \ x_g]\cZ_{0j}
([x_1\ \cdots \ x_g]_{j+1})^T+q(x)\\
&=&\frac{1}{2}\sum_{j=0}^\ell [x_1\ \cdots \ x_g]\cZ_{0j}\Pi_j^{-1}
\begin{bmatrix}x_1 \\ \vdots \\ x_g\end{bmatrix}_{j+1}+q(x),
\end{eqnarray*}
where  $\ell=d-2$ and $\textup{degree}\,q(x)\le 1$. Therefore, since
$\cZ_{0j}=Z_{0j}(0)$ is independent of $x$,
$$
p^\prime(x)[h]=I+II+q^\prime(x)[h],
$$
where
$$
I=\frac{1}{2}\sum_{j=0}^\ell [h_1\ \cdots \ h_g]\cZ_{0j}
\Pi_j^{-1}
\begin{bmatrix}x_1 \\ \vdots \\ x_g\end{bmatrix}_{j+1}=\sum_{j=1}^{\ell+1}
u_j^TV_j(x)[h]
$$
for some choice of vectors $u_j\in\RR^{g^{j+1}}$,
$$
q^\prime(x)[h]=[h_1\ \cdots \ h_g]u_0=u_0^TV_0(x)[h]\quad\text{for some
vector}\ u_0\in \RR^g
$$
and
\begin{eqnarray*}
II&=&\frac{1}{2}\sum_{j=0}^\ell [x_1\ \cdots \ x_g]\cZ_{0j}\Pi_j^{-1}
\left\{
\begin{bmatrix}h_1 \\ \vdots \\ h_g\end{bmatrix}\otimes
\begin{bmatrix}x_1 \\ \vdots \\ x_g\end{bmatrix}_j+\cdots +
\begin{bmatrix}x_1 \\ \vdots \\ x_g\end{bmatrix}_j\otimes
\begin{bmatrix}h_1 \\ \vdots \\ h_g\end{bmatrix}\right\}\\
&=&\frac{1}{2}[x_1\ \cdots \ x_g]\sum_{j=0}^\ell \cZ_{0j}\Pi_j^{-1}\sum_{i=0}^j
\theta_{j, j-i}\Pi_iV_i,
\end{eqnarray*}
where, with the help of (\ref{eq:vja}), the factor $\theta_{ji}$ can be
written as
$$
\theta_{ji}=\begin{bmatrix}x_1 \\ \vdots \\ x_g\end{bmatrix}_i\otimes
I_{g^{j+1-i}}\quad\text{for}\ i=1,\ldots, j \quad\text{and}\quad
\theta_{j0}=I_{g^{j+1}}.
$$
Thus, the double sum
$$
\sum_{j=0}^\ell \cZ_{0j}\Pi_j^{-1}\sum_{i=0}^j
\theta_{j, j-i}\Pi_iV_i=\sum_{i=0}^\ell \left(\sum_{j=i}^\ell \cZ_{0j}
\Pi_j^{-1}\theta_{j, j-i}\Pi_i\right)V_i.
$$
But the inner sum $\sum_{j=i}^\ell \cZ_{0j}\Pi_j^{-1}\theta_{j, j-i}\Pi_i$ can be
reexpressed in terms of the matrix polynomials
$$
K_j(x)=\Pi_{j+1}^{-1}\left(\begin{bmatrix}x_1\\ \vdots\\ x_g\end{bmatrix}
\otimes I_{g^{j+1}}\right)\Pi_j
$$
and their products
$$
K_{j+1}K_j=\Pi_{j+2}^{-1}\left(\begin{bmatrix}x_1\\ \vdots\\ x_g\end{bmatrix}_2
\otimes I_{g^{j+1}}\right)\Pi_j, \quad \cdots
$$
as
\begin{equation}
\label{eq:oct20a8}
\cZ_{0i}+
\cZ_{0,i+1}K_i+\cZ_{0,i+2}K_{i+1}K_i
+\cdots+\cZ_{0\ell}K_{\ell-1}\cdots K_i = Z_{0i}(x);
\end{equation}
and the last identity follows from the formula $Z(x)A(x)=\cZ$ in
Theorem 7.3  of \cite{DHMjda}. Thus,
$$
p^\prime(x)[h]=\frac{1}
{2}
[x_1\ \cdots\ x_g]\sum_{i=0}^\ell Z_{0i}(x)
V_i(x)[h],
$$
which, upon comparison with formula (\ref{eq:N1}), implies that
$$
\psi_i(x)^T=u_i^T+\frac{1}{2}[x_1\ \cdots\  x_g]Z_{0i}(x)\quad\text{for}\ i=0,
\ldots,\ell+1,
$$
since $Z_{0,\ell+1}(x)=0$. Therefore, $\psi_i(0)^T=u_i^T$ and the proof is
complete.
\end{proof}

\begin{cor}
If $\psi_j(0)=0$ for $j=0, \ldots , \ell$ in the setting of Theorem
\ref{thm:N1}, then
\begin{equation}
\label{eq:N3}
W_{ij}=Z_{i0}QZ_{0j}\quad\textrm{for}\quad i,j=0,\ldots,\ell\,,
\end{equation}
where $\ell=d-2$ and
\begin{equation}
\label{eq:N4}
Q=\frac{1}{4}\left[\begin{array}{c}x_1\\\vdots\\x_g\end{array}\right]
[x_1 \ \cdots \ x_g]\,.
\end{equation}
\end{cor}

\begin{theorem}
\label{thm:N2}
Let $p_k(x)=p(x_1, \ldots , x_g)$ be a homogeneous
symmetric nc polynomial of degree $k\ge 1$ in $g$ symmetric variables and
let
\begin{equation}
\label{eq:oct20b8}
p_k^\prime(x)[h]=\sum_{s=0}^{k-1}\psi_{ks}(x)^TV_s(x)[h].
\end{equation}
Then:
\begin{enumerate}
\item[\rm(1)] $\psi_{ks}(x)$ is a homogeneous polynomial of degree $k-1-s$ for
$s=0, \ldots , k-1$.
\item[\rm(2)] $\psi_{ks}(0)=0$ for $s=0, \ldots , k-2$ when $k\ge 2$.
\item[\rm(3)] Formulas (\ref{eq:N3}) and (\ref{eq:N4}) are in force for
$i, j=0, \ldots , k-2$ when $k\ge 2$.
\end{enumerate}
\end{theorem}

\begin{proof}
Assertion (1) is immediate from (\ref{eq:oct20b8}), since
$\psi_{ks}(x)^TV_s(x)[h]$ is a
homogeneous polynomial of degree $k-1$ in $x$ and $V_s$ is a homogeneous
polynomial of degree
$s$ in $x$. Assertions (2) and (3) then follow easily from (1)
and the preceding corollary. \end{proof}

Theorem \ref{thm:N2} also yields conclusions for nonhomogeneous polynomials,
subject to some restrictions.

\begin{cor}
\label{cor:may8a}
Let $p=p(x)=p(x_1, \ldots , x_g)$ be a
symmetric nc polynomial of degree $d\ge 2$ in $g$ symmetric variables
such that there are no terms of degree one and no terms of degree $d-1$ in
$p(x)$, let $\ell=d-2$  and let $\psi_s(x)^T$ denote the row vector
polynomials defined by
formula (\ref{eq:N1}). Then  $\psi_0(0)=0$ and $\psi_{\ell}(0)=0$
\end{cor}

\begin{proof}
Let $p(x)=\sum_{k=0}^d c_kp_k(x)$, where $p_k(x)$ is a homogeneous polynomial
of degree $k$ and $c_k=0$ or $c_k=1$. Then, by formula (\ref{eq:oct20b8}),
\begin{eqnarray*}
p^{\prime}(x)[h]&=&\sum_{k=1}^d c_kp_k^{\prime}(x)[h]\\
&=&\sum_{k=1}^d c_k\sum_{s=0}^{k-1}\psi_{ks}(x)^TV_s(x)[h]\\
&=&\sum_{s=1}^{d-1}\left\{\sum_{k=s+1}^dc_k\psi_{ks}(x)^T\right\}V_s(x)[h]\,.
\end{eqnarray*}
Thus, in formula (\ref{eq:N1}),
$$
\psi_0(x)^T=0\quad\textrm{and}\quad
\psi_s(x)^T=\sum_{k=s+1}^dc_k\psi_{ks}(x)^T\quad\textrm{for}\ s=1,\ldots,d-1\,.
$$
In particular,
$$
\psi_{\ell}(x)^T=\sum_{k=\ell+1}^dc_k\psi_{k\ell}(x)^T=c_{\ell+1}
\psi_{\ell+1,\ell}(x)^T+c_{\ell+2}
\psi_{\ell+2,\ell}(x)^T.
$$
Consequently,
$$
\psi_{\ell}(0)^T=c_{\ell+1}\psi_{\ell+2,\ell}(0)^T=0\,,
$$
since $c_{\ell+1}=0$, by assumption and $\psi_{\ell+2,\ell}(0)^{T}=0$ by Theorem
\ref{thm:N2}.
\end{proof}


\newpage
\vspace{-5.0in}
\centerline{NOT FOR PUBLICATION}

\tableofcontents
\newpage

\end{document}